\documentclass[11pt,reqno,tbtags]{amsart}

\title[Quasi-random graphs, subgraph counts and graph limits]
{More on quasi-random graphs, subgraph counts and graph limits}

\date{26 May, 2014}

\author{Svante Janson}
\thanks{SJ partly supported by the Knut and Alice Wallenberg Foundation}
\address{Department of Mathematics, Uppsala University, PO Box 480,
SE-751~06 Uppsala, Sweden}
\email{svante.janson@math.uu.se}
\newcommand\urladdrx[1]{{\urladdr{\def~{{\tiny$\sim$}}#1}}}
\urladdrx{http://www2.math.uu.se/~svante/}

\author{Vera T. S\'os}
\thanks{VTS partly supported by OTKA 101535.2012}
\address{A. R\'enyi Institute of Mathematics, Budapest, Hungary}
\email{t.sos.vera@renyi.mta.hu}

\usepackage{amsxtra}
\usepackage{amssymb}
\usepackage{url}
\usepackage[square,numbers]{natbib}
\bibpunct[, ]{[}{]}{;}{n}{,}{;}

\numberwithin{equation}{section}

\renewcommand\le{\leqslant}
\renewcommand\ge{\geqslant}

\allowdisplaybreaks

\newtheorem{theorem}{Theorem}[section]
\newtheorem{lemma}[theorem]{Lemma}

\newtheorem{conjecture}[theorem]{Conjecture}

\theoremstyle{definition}
\newtheorem{example}[theorem]{Example}
\newtheorem{definition}[theorem]{Definition}
\newtheorem{definitions}[theorem]{Definition}
\newtheorem{problem}[theorem]{Problem}
\newtheorem{remark}[theorem]{Remark}

\newtheorem*{ack}{Acknowledgement}

\theoremstyle{remark}

\newenvironment{romenumerate}[1][0pt]{
\addtolength{\leftmargini}{#1}\begin{enumerate}
 }{\end{enumerate}}

\newenvironment{alphenumerate}[1][0pt]{
\addtolength{\leftmargini}{#1}\begin{enumerate}
 }{\end{enumerate}}

\newcounter{oldenumi}
{\setcounter{oldenumi}{\value{enumi}}
\begin{romenumerate} \setcounter{enumi}{\value{oldenumi}}}
{\end{romenumerate}}

\newcounter{thmenumerate}
\newenvironment{thmenumerate}
{\setcounter{thmenumerate}{0}%
 \def\item{\par
 \refstepcounter{thmenumerate}\textup{(\roman{thmenumerate})\enspace}}
}
{}

\newcounter{romxenumerate}   

\newcounter{xenumerate}   

\newcommand\pfitemx[1]{\par#1:}
\newcommand\pfitemref[1]{\pfitemx{\ref{#1}}}

\newcounter{steps}
\newcommand\stepx{\smallskip\noindent\refstepcounter{steps}%
 \emph{Step \arabic{steps}:}\noindent}

\newcommand{\refT}[1]{Theorem~\ref{#1}}

\newcommand{\refL}[1]{Lemma~\ref{#1}}
\newcommand{\refR}[1]{Remark~\ref{#1}}
\newcommand{\refS}[1]{Section~\ref{#1}}

\newcommand{\refP}[1]{Problem~\ref{#1}}
\newcommand{\refD}[1]{Definition~\ref{#1}}
\newcommand{\refE}[1]{Example~\ref{#1}}

\newcommand{\refConj}[1]{Conjecture~\ref{#1}}

\begingroup
  \divide\count255 by 60
  \multiply\count255 by -60
  \advance\count255 by \time
  \ifnum \count255 < 10 \xdef\klockan{\the\count1.0\the\count255}
  \else\xdef\klockan{\the\count1.\the\count255}\fi
\endgroup

\newcommand\nopf{\qed}   

\newcommand{\sumim}{\sum_{i=1}^m}

\newcommand{\sumir}{\sum_{i=1}^r}

\newcommand\set[1]{\ensuremath{\{#1\}}}

\newcommand\xpar[1]{(#1)}
\newcommand\bigpar[1]{\bigl(#1\bigr)}
\newcommand\Bigpar[1]{\Bigl(#1\Bigr)}

\def\rompar(#1){\textup(#1\textup)}    
\newcommand\uppar[1]{\textup(#1\textup)}    

\newcommand\Bigparfrac[2]{\Bigpar{\frac{#1}{#2}}}

\def\xexp(#1){e^{#1}}

\newcommand\floor[1]{\lfloor#1\rfloor}

\newcommand\nn{[n]}
\newcommand\ntoo{\ensuremath{{n\to\infty}}}

\newcommand\punkt{.\spacefactor=1000}    
    
\newcommand\ie{i.e\punkt}
\newcommand\eg{e.g\punkt}
\newcommand\viz{viz\punkt}
\newcommand\cf{cf\punkt}

\newcommand{\aex}{a.e\punkt}

\newcommand\ii{\mathrm{i}}

\newcommand\bbR{\mathbb R}
\newcommand\bbC{\mathbb C}

\newcommand\bbZ{\mathbb Z}

\newcounter{CC}
\newcounter{cc}

\newcommand\supp{\operatorname{supp}}

\newcommand\aut{\operatorname{aut}}

\newcommand\ga{\alpha}
\newcommand\gb{\beta}

\newcommand\gf{\varphi}
\newcommand\gam{\gamma}

\newcommand\gl{\lambda}
\newcommand\gL{\Lambda}

\newcommand\gs{\sigma}

\newcommand\eps{\varepsilon}

\renewcommand\phi{\xxx}  

\newcommand\cB{\mathcal B}

\newcommand\cP{\mathcal P}

\newcommand\etta{\boldsymbol1}

\newcommand\qw{^{-1}}

\newcommand\intoi{\int_0^1}

\newcommand\oi{[0,1]}

\newcommand\setoi{\set{0,1}}

\newcommand\dd{\,\mathrm{d}}

\newcommand\lhs{left-hand side}

\newcommand\aaf{A_1,\dots,A_\ff}
\newcommand\aafx{A_1\times\dots\times A_\ff}
\newcommand\aam{A_1,\dots,A_m}
\newcommand\aamx{A_1\times\dots\times A_m}

\newcommand\aar{A_1,\dots,A_r}
\newcommand\garr{\ga_1,\dots,\ga_r}

\newcommand\fAAm{f^{A_2,\dots, A_m}}

\newcommand\uuf{U_1,\dots,U_\ff}
\newcommand\uum{U_1,\dots,U_m}
\newcommand\uur{U_1,\dots,U_r}

\newcommand\xxf{x_1,\dots,x_\ff}
\newcommand\xxm{x_1,\dots,x_m}
\newcommand\xxr{x_1,\dots,x_r}
\newcommand\xim[1]{#1_{i_1},\dots,#1_{i_m}}
\newcommand\yym{y_1,\dots,y_m}

\newcommand\xxgsm{x_{\gs(1)},\dots,x_{\gs(m)}}
\newcommand\ff{{|F|}}
\newcommand\eff{^{e(F)}}
\newcommand\prodif{\prod_{i=1}^\ff}
\newcommand\prodim{\prod_{i=1}^m}
\newcommand\prodjm{\prod_{j=1}^m}

\newcommand\intaafx{\int_{\aafx}}
\newcommand\intaamx{\int_{\aamx}}

\newcommand\tPsiq{\widetilde\Psi}
\newcommand\tPsiqq{\widetilde\Psi^{*}}
\newcommand\fS{\mathfrak S}

\newcommand\psifx[1]{\Psi_{F,#1}}
\newcommand\psifw{\psifx W}

\newcommand\tpsifx[1]{\widetilde\Psi_{F,#1}}
\newcommand\tpsifw{\tpsifx W}

\newcommand\tpsiqfx[1]{\tPsiq_{F,#1}}
\newcommand\tpsiqqfw{\tPsiqq_{F,W}}

\newcommand\tpsiqfw{\tpsiqfx W}

\newcommand\phiw{\Phi_W}
\newcommand\Phix{\widehat\Phi} 

\newcommand\mpb{measure preserving bijection}
\newcommand\pqr{$p$-\qr}

\newcommand\qr{quasi-random}

\newcommand\gnq{\ensuremath{(G_n)}}

\newcommand\gax{\gam}
\newcommand\gc{\ga}

\newcommand\xij{_{ij}}

\newcommand\tN{\widetilde N}
\newcommand\uufgs{U_{\gs(1)},\dots,U_{\gs(\ff)}}
\newcommand\qrp{\qr{} property}
\newcommand\qrpp{\qr{} properties}
\newcommand\pqrp{\pqr{} property}
\newcommand\gamm{\ga_1,\dots,\ga_m}
\newcommand\ctP{\widetilde\cP}
\newcommand\cPu{\cP'}
\newcommand\cPl{\cP}
\newcommand\gay{\ga,\dots,\ga}
\newcommand\cplw{\cP_*}
\newcommand\cpuw{\cP'_*}
\newcommand\ctpw{\ctP_*}
\newcommand\mmm{1/m,\dots,1/m}
\newcommand\eef{^{e(F)}}
\newcommand\gluvs{\gL_{F;u,v,s}}
\newcommand\gLx{\hat\Lambda}
\newcommand\glxuvs{\gLx_{F;u,v,s}}

\newcommand\aamq{\set{\aam}}
\newcommand\nnm{n_1,\dots,n_m}
\newcommand\boc{B_0\comp}
\newcommand\comp{^{\mathsf{c}}}
\newcommand\hf{\widehat f}
\newcommand\tf{\tilde f}
\newcommand\hg{\widehat g}
\newcommand\hh{\widehat h}

\newcommand\bh{\bar h}
\newcommand\bW{\overline W}

\newcommand\sumii{\sum_{i_1<\dots<i_m}}
\newcommand\sumiir{\sum_{i_1<\dots<i_m\le r}}
\newcommand\symm{_{\textsf s}}
\newcommand\xN{\widehat N}
\newcommand\ctxP[1]{\ctP_{#1}}
\newcommand\cxxP[1]{\widehat P_{#1}}
\newcommand\cPz{\cP_{2,1}}
\newcommand\ctPz{\ctxP{2,1}}

\newcommand{\Lovasz}{Lov\'asz}

\newcommand{\maple}{\texttt{Maple}}

\begin{document}

\subjclass[2010]{05C99} 

\begin{abstract} 
We study some properties of graphs (or, rather, graph sequences)
defined by demanding that the number of
subgraphs of a given type, with vertices in subsets of given sizes,
approximatively equals the number expected in a random graph.
It has been shown by several authors 
that several such conditions are quasi-random, but that there are
exceptions. 
In order to understand this better, we investigate some new
properties of this type. We show that these properties too are quasi-random,
at least in some cases; however, there are also cases that are left as open
problems, and we discuss why the proofs fail in these cases. 

The proofs are based on the theory of graph limits; 
and on the method and results developed by Janson (2011),
this translates
the combinatorial problem to an analytic problem, which then is translated to an
algebraic problem.
\end{abstract}

\maketitle

\section{Introduction}\label{S:intro}

Consider a sequence of graphs $(G_n)$, with $|G_n|\to\infty$ as \ntoo.
\citet{Thomason87a,Thomason87b} and \citet{ChungGW:quasi} showed
that a number of different 'random-like' properties of 
the sequence $(G_n)$
are equivalent, and we say that \gnq{} is \emph{\qr}, or more precisely
\emph{\pqr}, if 
it satisfies these properties. (Here $p\in\oi$ is a
parameter.)
Many other equivalent properties of different types
have later been added by various authors. We say that a property of
sequences $(G_n)$ of graphs (with $|G_n|\to\infty$) is a 
\emph{\qrp} (or more specifically a \emph{\pqrp}) if it 
characterizes \qr{} (or \pqr) sequences of graphs.

One of the \qrpp{} considered by \citet{ChungGW:quasi}
is based on subgraph counts, see \eqref{pqr} below.
Further \qrpp{} based on restricted 
subgraph count properties have been found by 
\citet{ChungG},
\citet{SS:nni,SS:ind},
\citet{Shapira},  \citet{ShapiraYuster,ShapiraYuster:hyper},
\citet{Yuster},
\citet{SJ234}, 
\citet{HuangLee}, 
see \refS{Smain}.

The purpose of the present paper is to continue the study of such properties
by considering some further cases not treated earlier;
in particular (Theorems \ref{TA} and \ref{T=}), we prove that some further
properties of this type are \qr. Our main purpose is not to just add to the
already long list of \qr{} properties; we hope that this study will
contribute to the understanding of this type of \qr{} properties, 
and in particular explain why the case in \refT{T=} is more difficult than
the one in \refT{TA}. (See also \refS{Sless} for a discussion of further
similar properties.)

We use the method of \citet{SJ234} based on graph limits.
We assume that the reader is familiar with the basics of the theory of
graph limits and graphons developed in \eg{} \citet{LSz} and \citet{BCLSV1};
otherwise, see 
\citet{SJ234} (for the present context) or the comprehensive book by 
\citet{Lovasz}.
As is well-known, there is a simple characterization of \qr{} sequences in
terms of graph limits:
a sequence $(G_n)$ with $|G_n|\to\infty$ is \pqr{} if and only if
$G_n\to W_p$, where $W_p$ is the graphon that is constant with $W_p=p$
\cite{BCLSV1,BCLSV2,LSz},
see also \cite[Section 1.4.2 and Example 11.37]{Lovasz}.
(Indeed, quasi-random graphs form one of the roots of graph limit theory.) 

The idea of the method is to use this characterization
to
translate the property of graph sequences
to a property of graphons, and then show that only constant graphons satisfy
this property.
It turns out that this leads to both analytic (\refS{Sanal})
and algebraic (\refS{Salg}) problems, which
we find interesting in themselves.
We have only partly succeeded to solve these problems, so we leave several
open problems. 

\begin{remark}
  Many of the references above use Szemer\'edi's regularity lemma 
as their main tool to study quasi-random properties; it has been known since
\cite{SS:Sze} 
  that quasi-randomness can be characterized using Szemer\'edi partitions.
It is also well-known that there are strong connections between
Szemer\'edi's regularity lemma and graph limits, 
see \cite{BCLSV1,LSz:Sz,Lovasz}, 
so on a deeper level the methods are related
although they superficially look very different. (It thus might be possible to
translate arguments of one type to the other, although it is far from clear
how this might be done.)
Both methods lead also to
the same (sometimes difficult) algebraic problems. 
As discussed in \cite{SJ234}, the method used here
eliminates the many small error terms in the regularity lemma approach; on
the other 
hand, it leads to analytic problems with no direct counterpart in the other
approach.
It is partly a matter of taste what type of
arguments one prefers. 
\end{remark}

\begin{ack}
This research was begun 
during the workshop 
\emph{Graph limits, homomorphisms and structures II} at
Hrani{\v c}n{\'i} Z{\'a}me{\v c}ek,  Czech Republic, 2012;
parts were also done during the workshop
\emph{Combinatorics and Probability}
at Mathematisches Forschungsinstitut Oberwolfach, 2013.
We thank the organisers for providing us with these opportunities.
\end{ack}

\section{Notation, background and main results}\label{Smain}

All graphs in this paper are finite, undirected and simple.
The vertex and edge sets of a graph $G$ are denoted by
$V(G)$ and $E(G)$. We write $|G|:=|V(G)|$ for the
number of vertices of $G$, and $e(G):=|E(G)|$ for the
number of edges.
As usual, $\nn:=\set{1,\dots,n}$.

All unspecified limits in this paper are as \ntoo, and $o(1)$
denotes a quantity that tends to $0$ as \ntoo.
We will often use $o(1)$
for quantities that depend on some subset(s)
of a vertex set $V(G)$; 
we then always
implicitly assume that the convergence is uniform for all choices of
the subsets. 
We interpret $o(a_n)$ for a given sequence $a_n$ similarly.

Let $F$ and $G$ be labelled graphs.
For convenience, we assume throughout the paper (when it matters) that
$V(F)=[\ff]=\set{1,\dots,\ff}$. We generally let $m=|F|$.

\begin{definitions}\label{DN}
  \begin{thmenumerate}
  \item 
$N(F,G)$ is the number of labelled copies of $F$ in $G$
(not necessarily induced); 
equivalently, $N(F,G)$ is the
number of injective maps $\gf:V(F)\to V(G)$ that are graph
homomorphisms
(\ie, if $i$ and $j$ are adjacent in $F$, then
$\gf(i)$ and $\gf(j)$ are adjacent in $G$).

\item 
If $\uuf$ are subsets of $V(G)$, 
let
$N(F,G;\uuf)$ be the number of 
labelled copies of $F$ in $G$ with the $i$th vertex
in $U_i$; equivalently, $N(F,G;\uuf)$ is the number of
injective graph homomorphisms 
$\gf:F\to G$ such that $\gf(i)\in U_i$ for every 
$i\in V(F)$. 
(Note that we  consider a fixed labelling of the vertices of $F$
and count the number of copies where vertex $i$ is in $U_i$, so the
labelling and the ordering
of $\uuf$ are important.)

\item 
We  also define a symmetrized version 
$\tN(F,G;\uuf)$ 
by taking the average over all labellings of $F$; equivalently,
\begin{equation}\label{tN}
  \tN(F,G;\uuf):=\frac{1}{|F|!}\sum_{\gs}N(F,G;\uufgs),
\end{equation}
summing over all permutations $\gs$ of $\set{1,\dots,\ff}$.
  \end{thmenumerate}
\end{definitions}

In (ii) and (iii), 
we are often interested in the case when $\uuf$ are pairwise disjoint, 
and then $\tN(F,G;\uuf)$ is 
the number of labelled copies of $F$ in $G$ with one
vertex in each set $U_i$ (in any order),
divided by $1/|F|!$.

\begin{remark}\label{R=}
If either $U_1=\dots=U_{|F|}$ or $F=K_m$  for some $m$, then
$  \tN(F,G;\uuf):=N(F,G;\uuf)$, and the symmetrized  version $\tN$ is
equal to $N$.
\end{remark}

One of the several equivalent definitions of \qr{} graphs by
\citet{ChungGW:quasi} is the following using the subgraph counts $N(F,G)$:
\begin{theorem}[\citet{ChungGW:quasi}]
A sequence of graphs
\gnq{} with $|G_n|\to\infty$
is \pqr{} if and only if, 
for every graph $F$,
\begin{equation}
  \label{pqr}
N(F,G_n)= (p\eff+o(1)) |G_n|^\ff.
\end{equation}
\qed
\end{theorem}

It is not necessary to require \eqref{pqr}
for all graphs $F$; in particular, it suffices to use the graphs $K_2$ and
$C_4$ \cite{ChungGW:quasi}. However, it is not enough to require \eqref{pqr}
for just one graph $F$. As a substitute, \citet{SS:nni} considered
the hereditary version of \eqref{pqr}, \ie{} the condition
$N(F,G;U,\dots,U)$ for subsets $U$.

We note first that
for quasi-random graphs,
it is shown in \cite{SS:nni} and \cite{Shapira} that 
the restricted subgraph count
$N(F,G;U_1,\dots,U_\ff)$ is 
asymptotically the same as it is  for random graphs, for any subsets
$U_1,\dots,U_\ff$.
(For a proof 
using graph limits, see \citet[Lemma 4.2]{SJ234}.)

\begin{lemma}[\cite{SS:nni} and \cite{Shapira}]
  \label{L0}
Suppose that $(G_n)$ is a \pqr{} sequence of graphs, 
where $0\le p\le 1$, and
let $F$ be any fixed graph with $e(F)>0$.
Then,
for all subsets $\uuf$ of\/ $V(G_n)$,
\begin{equation}\label{l0}
  N(F,G_n;\uuf)=p^{e(F)}\prodif|U_i|+o\bigpar{|G_n|^\ff}.
\end{equation}
and
\begin{equation}\label{l0t}
  \tN(F,G_n;\uuf)=p^{e(F)}\prodif|U_i|+o\bigpar{|G_n|^\ff}.
\end{equation}
\qed
\end{lemma}

Note that \eqref{l0t} is an immediate consequence of \eqref{l0} by the
definition \eqref{tN}. 

Conversely, \citet{SS:nni} showed that \eqref{l0} implies that $\gnq$ is
\pqr. Actually, they considered only the symmetric case $U_1=\dots=U_\ff$
and proved the following stronger result. (In this case, \eqref{l0t} is
obviously equivalent to \eqref{l0}, see \refR{R=}.)

\begin{theorem}[\citet{SS:nni}]
  \label{TSS}
Suppose that $(G_n)$ is a sequence of graphs with $|G_n|\to\infty$.
Let $F$ be any fixed graph with $e(F)>0$ and let $0<p\le1$.
Then $(G_n)$ is \pqr{} if and only if,
for all subsets $U$ of\/ $V(G_n)$,
\eqref{l0} holds with $U_1=\dots=U_\ff=U$.
\nopf
\end{theorem}

\begin{remark}
The case $F=K_2$, when $N(K_2,G_n;U)$ is twice the number of edges
with both endpoints in $U$, is one of the original \qr{} properties in
\citet{ChungGW:quasi}.  
\end{remark}

\begin{remark}\label{RK3bi}
\refT{TSS} obviously fails when $e(F)=0$, since then
\eqref{l0} holds trivially for any $G_n$. It fails also if
  $p=0$;
for example, if $F=K_3$ and $G_n$ is the complete
  bipartite graph $K_{n,n}$.
\end{remark}

In other words, \refT{TSS} says that, if $e(F)>0$ and $0<p\le1$,
then \eqref{l0} and \eqref{l0t} (for arbitrary $\uuf$) are
both {\pqr{} properties},
and this holds also if we restrict $\uuf$ to $U_1=\dots=U_\ff$.

Several authors have considered other restrictions on $\uuf$ and shown that 
\eqref{l0} or \eqref{l0t} still is a \qrp.

\citet{Shapira} and \citet{Yuster} continued to consider $U_1=\dots=U_\ff$,
and assumed further that $|U_1|=\floor{\ga|G_n|}$ for some fixed $\ga$ with
$0<\ga<1$; they showed (\cite{Shapira} for $\ga=1/(\ff+1)$ and \cite{Yuster}
in general) that \eqref{l0} for such $\uuf$ is a \qrp.
(The case $F=K_2$ and $\ga=1/2$ is in 
\citet{ChungGW:quasi}.)  
Note that for such $\uuf$, \eqref{l0t} is equivalent to \eqref{l0} by \refR{R=}.

The case when $\uuf$ are disjoint and furthermore have the same size
is considered by \citet{Shapira} and \citet{ShapiraYuster};
they show that \eqref{l0t} with this restriction also is a \qrp.
(As a consequence, \eqref{l0} with this restriction is a \qrp.)
Moreover, by combining \citet[Lemma 2.2]{Shapira} and the result of
\citet{Yuster} just mentioned, 
it follows that it suffices to consider disjoint $\uuf$ with the same size 
$\floor{\ga|G_n|}$, for any fixed $\ga<1/|F|$.

We introduce some more notation. 

\begin{definitions}\label{DP}
Let $F$ be a graph, 
$m:=\ff$ and  $(\gamm)$  a vector of positive 
numbers with $\sumim\ga_i\le1$; let further $p\in\oi$.
We define the following properties
of graph sequences $\gnq$.
(For convenience, we omit $p$ from the notations.)
  \begin{romenumerate}[-10pt]
  \item \label{dp1}
Let $F$ be labelled.
Then $\cPl(F;\gamm)$ is the property
that \eqref{l0} holds for all disjoint
subsets $\uum$ of $V(G_n)$ with $|U_i|=\floor{\ga_i|G_n|}$, $i=1,\dots,m$.

\item 
Let $F$ be unlabelled. Then
$\cPu(F;\gamm)$ is the property
that 
$\cPl(F;\gamm)$ holds for every labelling of $F$.

\item 
Let $F$ be unlabelled. Then
$\ctP(F;\gamm)$ is the property that \eqref{l0t} holds for all
$\uum$ as in \ref{dp1}. 
  \end{romenumerate}
\end{definitions}

Of course, we can use $\cPu$ and $\ctP$ also for a labelled $F$ by ignoring
the labelling.

\begin{remark}\label{RPPP}
If $F=K_m$, then all labellings of $F$ are equivalent, and the three
properties $\cPl(F;\gamm)$, $\cPu(F;\gamm)$ and $\ctP(F;\gamm)$ are
equivalent.
In general, 
$\cPu(F;\gamm)\implies\ctP(F;\gamm)$ by the definition of $\tN$ as an
average of $N$ over all labellings of $F$, but we do not know whether the
converse implication always holds. 

Furthermore, for a fixed labelling of $F$, $\cPu(F;\gamm)$ is equivalent to 
the conjunction of $\cPl(F;\ga_{\gs(1)},\dots,\ga_{\gs(m)})$ for all
permutations 
$(\ga_{\gs(1)},\dots,\ga_{\gs(m)})$ of $(\gamm)$. In particular, if
$\ga_1=\dots=\ga_m$, then $\cPu(F;\gamm)$ equals $\cP(F;\gamm)$, for any
labelling. 

In general, trivially $\cPu(F;\gamm)\implies\cP(F;\gamm)$
for a labelled graph $F$,
but we do not
know whether the converse holds. Nor do we know any general implications
between 
$\cPl(F;\gamm)$ and $\ctP(F;\gamm)$.

See further \refR{RX}.
\end{remark}

Using this notation,
it thus follows 
from \citet{Shapira} and \citet{Yuster} that, 
for any graph $F$ with $e(F)>0$ and $0<p\le1$,
$\ctP(F;\gay)$
is a \qrp{} for every $\ga<1/\ff$.
This can also be
proved by the methods of \citet{SJ234}, where the somewhat weaker  statement
that 
$\cP(F;\gay)$ is a \qrp{} 
for every $\ga<1/\ff$
is shown \cite[Theorem 3.6]{SJ234}.
We show here a more general statement in \refT{TA} below.

\begin{example}\label{ERK2}
  For $F=K_2$, $\ctP(K_2,\ga_1,\ga_2)=\cP(K_2,\ga_1,\ga_2)$ says
that (asymptotically) the number of edges $e(U_1,U_2)$ is 
as expected in $G(n,p)$ for any two
disjoint sets $U_1,U_2$ with $U_i=\floor{\ga_i|G_n|}$.
\citet{ChungG} showed that the cut property $\cP(K_2;\ga,1-\ga)$ is a \qrp{}
for every fixed $\ga\in(0,1)$ except $\ga=1/2$, when it is not;
see further \citet[Section~9]{SJ234}.
\citet{SS:nni} showed that $\cP(K_2,1/3,1/3)$ is a \qrp.
\end{example}

\citet[Proposition 14]{ShapiraYuster:hyper} showed (as a consequence of
related results for cuts in hypergraphs) that $\cP(K_m,\ga_1,\dots,\ga_m)$
is a \qrp, for every  
$m\ge2$ and $(\ga_1,\dots,\ga_m)\neq(1/m,\dots,1/m)$ with $\sumim \ga_i=1$.
This can easily be extended to subgraph counts for arbitrary graphs $F$ with
$e(F)>0$; we give a proof using our methods in \refS{Spf2}.

\begin{theorem}\label{TA}
  Let $F$ be a graph with $e(F)>0$, 
and let $0<p\le1$.
Further, let  $(\gamm)$ be a vector of positive numbers of length $m=\ff$ with
$\sumim\ga_i\le1$. 
\begin{romenumerate}[-10pt]
\item \label{ta1}
If\/ $(\gamm)\neq(1/m,\dots,1/m)$,
then 
$\ctP(F;\gamm)$ and the stronger $\cPu(F;\gamm)$ are \qrpp. 
\item \label{ta2}
If\/ $\sumim \ga_i<1$,
then 
$\cPl(F;\gamm)$ is a \qrp. 
\end{romenumerate}
\end{theorem}

The exceptional case $\ga_1=\dots=\ga_m=1/m$
is more complicated; \citet{ShapiraYuster:hyper} showed
that the related hypergraph cut property used by them to prove \refT{TA}
fails in this case; 
nevertheless, \citet{HuangLee} showed that also
$\cP(K_m,1/m,\dots,1/m)$ is a \qrp{} for any $m\ge3$. (For $m=2$ it is  not,
see \refE{ERK2}.)

We give a new proof of their theorem in \refS{Spf1} and extend the result to 
counts of several other subgraphs. With our methods using graph
limits, the crucial fact is that while the central analytic
\refL{L3} does not
generalize to the case $(\gamm)=(\mmm)$, there is a weaker version \refL{Lmain}
that holds in this case, and this is sufficient to draw the conclusion
with some extra algebraic work. We have so far not succeded to extend
the final, algebraic, part to all graphs $F$, but
we can prove  the following, see \refS{Spf1}.
(\refS{Spf1} contains also some further examples of small
graphs $F$ for which the conclusion 
holds.)

\begin{theorem}\label{T=}
Let $F$ be a graph with $e(F)>1$ and  $m=|F|$.
Let also $0<p\le1$.
If $F$ is either a regular graph
or a star,
or disconnected,
then $\cP(F;\mmm)$ and the weaker $\ctP(F;\mmm)$ are \qr{} properties.
\end{theorem}

One indication that this theorem is more complicated than \refT{TA} is that 
the conclusion is false for $F=K_2$ by \refE{ERK2}, 
and slightly more generally when $e(F)\le1$.
We conjecture that this is the only counterexample.

\begin{conjecture}\label{Conj1}
  \refT{T=} holds for any graph $F$ with $e(F)>1$.
\end{conjecture}


\begin{remark}\label{RX}
When $F\neq K_m$, the relation between the properties $\cPl$ (non-averaged)
and $\ctP$ (averaged) is not completely clear. (For $F=K_m$, these
properties coincide, see \refR{RPPP}.) 

Consider first $\ga_1=\dots=\ga_m=1/m$ as in \refT{T=}. Then
$\cPl=\cPu\implies\ctP$. 
(See  \refR{RPPP} again.) 
For a graph $F$ such that \refT{T=} applies, the theorem implies that the properties
are equivalent, but as said above, we do not know whether that holds in general.
In principle, it should be easier to show that the property
$\cP(F;\mmm)$ is \pqr{} than to show that 
the weaker (averaged) property
$\ctP(F;\mmm)$ is; 
it is even conceivable that there exists a
counterexample to \refConj{Conj1} such that nevertheless $\cP(F;\mmm)$ is \pqr.
However, our method of proof uses \refL{Lmain} below which assumes that the
function $f$ there is symmetric, and  hence 
our proofs use the symmetric $\ctP(F;\mmm)$ and we are not able to use the extra
power of $\cP(F;\allowbreak\mmm)$. 
For example, we cannot answer the following question.
(Cf.\ \refS{Stwo} for $\ctP(F;\allowbreak\mmm)$.)
A \emph{2-type graphon} is a graphon that is constant on the sets 
$S_i\times S_j$, $i,j\in\set{1,2}$, for some partition $\oi=S_1\cup S_2$
into two disjoint sets;  we can without loss of generality assume that
the sets $S_i$ are intervals.
(Equivalently, we
may regard $W$ as a graphon defined on a two-point probability space.)

\begin{problem}\label{P2mmm}  
If $F$ is such that $\cP(F;\mmm)$ is not \pqr, 
is there always a 2-type graphon witnessing
this?
\end{problem}

For other sequences $\gamm$, we note first that if $\sumim\ga_i<1$, then
\refT{TA} shows that both $\cPl$ and $\ctP$ are \qrpp, and thus equivalent.
Similarly, if $\sumim\ga_i=1$ but $(\gamm)\neq(\mmm)$, then $\ctP$ is \qr{}
by \refT{TA}, and
thus $\ctP\implies\cPl$. However, we do not know whether the converse holds:
\begin{problem}
Suppose that $F$ is a labelled graph with $e(F)>0$, that $0<p\le1$ and that
$\sumim\ga_i=1$ but $(\gamm)\neq(\mmm)$. Is then $\cPl(F;\gamm)$ a \qrp?
\end{problem}

If there is any case such that the answer to this problem is negative, we can
ask the same question as in \refP{P2mmm}:

\begin{problem}\label{P2gamm} 
If $F$ and $(\gamm)$ are such that $\cP(F;\gamm)$ is not \pqr, 
is there always a 2-type graphon witnessing
this?
\end{problem}

\begin{example}\label{EP3}
  Let $F=P_3=K_{1,2}$,
for definiteness labelled with edges 12 and 13, and consider
  the property $\cPl(F;\ga_1,\ga_2,\ga_3)$. If $\ga_1+\ga_2+\ga_3<1$, then
  the property is \qr{} by \refT{TA}; thus assume $\ga_1+\ga_2+\ga_3=1$.
In the case $\ga_1=\ga_2=\ga_3=1/3$, the property is \qr{} by \refT{T=}.
We can show this also in the case $\ga_2\neq\ga_3$, using the symmetry of
$P_3$, see \refR{RP3b}. However, we do not know if this extends to
$\ga_2=\ga_3$, for example in the following case:
\end{example}

\begin{problem}\label{PP3}
  Is (with the labelling above) $\cPl(P_3,\frac12,\frac14,\frac14)$ a \qrp?
\end{problem}
\end{remark}

\begin{remark}\label{R3}
We have 
considered 
the subgraph counts
$N(F,G_n;\uum)$ and $\tN(F,G_n;\uum)$
in two cases: either $U_1=\dots=U_m$ (as in \cite{SS:nni})
or $U_1,\dots,U_m$ are disjoint.
It also seems interesting to consider other, intermediate, cases of
restrictions. This is discussed in \refS{Sless}, where we in particular
consider, as a typical example,
the case $U_1=U_2$ and $U_1\cap U_3=\emptyset$. 
\end{remark}

\begin{remark}
We consider in this paper not necessarily induced copies of a fixed graph
$F$.
There are also similar results for counts of induced copies of $F$, but
these are more complicated and less complete,
see \citet{SS:ind},
\citet{ShapiraYuster} and
\citet{SJ234}. 
We hope to return to the induced case, but leave it for now as an open problem:
\end{remark}

\begin{problem}
  Are there analogues of Theorems \ref{TA} and \ref{T=} for the induced case?
\end{problem}

\section{Transfer to graph limits}

We introduce some further notation:

The support of a function $\psi$ is the set
$\supp(\psi):= \set{x:\psi(x)\neq0}$.

$\gl$ denotes Lebesgue measure.

All functions are supposed to be (Lebesgue) measurable.

If $F$ is a labelled graph and $W$ a graphon, we define
\begin{equation}\label{psifw}
  \psifw(\xxf):=\prod_{ij\in E(F)} W(x_i,x_j).
\end{equation}

If $f$ is a function on $\oi^m$ for some $m$, we
let $\tf$ denote its symmetrization defined by
\begin{equation}
  \label{symm}
\tf(\xxm):=\frac1{m!}\sum_{\gs\in\fS_m} f\bigpar{\xxgsm},
\end{equation}
where $\fS_m$ is the symmetric group of all $m!$
permutations of \set{1,\dots,m}. 

The connection between the subgraph count properties and properties of graph
limits is given by the following lemma.

\begin{lemma}
  \label{LA1c}
Suppose that $G_n\to W$ for some graphon $W$. 
Let
$F$ be a fixed graph, let $m:=\ff$  and 
let $\gax\ge0$ and $\gamm\in(0,1)$ be fixed numbers
with $\sumim \ga_i\le1$.
Then the following are equivalent:
\begin{romenumerate}[-10pt]
  \item\label{La1ci}
For all disjoint subsets $\uuf$ of\/ $V(G_n)$
with $|U_i|=\floor{\ga_i|G_n|}$,
\begin{equation}\label{la1ci}
  N(F,G_n;\uuf)=\gax\prodif|U_i|+o\bigpar{|G_n|^\ff}.
\end{equation}
\item\label{La1cii}
For all disjoint subsets $\aaf$ of\/ $\oi$ with $\gl(A_i)=\gc_i$,
\begin{equation}\label{la1cii}
  \intaafx \psifw(\xxf)=\gax\prodif\gl(A_i).
\end{equation}
\end{romenumerate}
The same holds if we replace $N$ in \ref{La1ci}
and $\psifw$ in \ref{La1cii} by the symmetrized versions
$\tN$ and $\tpsifw$.
\end{lemma}
\begin{proof}
The case with $N$ and $\psifw$ and with 
$\ga_1=\dots=\ga_m <1/\ff$ is part of \citet[Lemma 7.2]{SJ234}.
The case of general $\gamm$, and the symmetrized version with $\tN$ and
$\tpsifw$  
are proved in exactly the same way.
\end{proof}

With this lemma in mind, we make the following definitions corresponding to
\refD{DP}. 

\begin{definition}
  \label{DW}
Let, as in \refD{DP},
$F$ be a graph,  $m:=\ff$,  $(\gamm)$  a vector of positive
numbers with $\sumim\ga_i\le1$, and $p\in\oi$. 
We define the following properties of graphons $W$.
\begin{romenumerate}[-10pt]
\item\label{dwl} 
$\cplw(F;\gamm)$ is the property that 
\begin{equation}\label{e33}
  \intaamx \psifw(\xxm)=p^{e(F)}\prodim\gl(A_i),
\end{equation}
for all disjoint
subsets $\aam$ of $\oi$ with $\gl(A_i)={\ga_i}$, $i=1,\dots,m$.
\item 
$\cpuw(F;\gamm)$ is the property that $\cplw(F;\gamm)$ holds for every
  labelling of $F$.
\item 
$\ctpw(F;\gamm)$ is the property that 
\begin{equation}\label{e34}
  \intaamx \tpsifw(\xxm)=p^{e(F)}\prodim\gl(A_i)
\end{equation}
for all  $\aam$ as in \ref{dwl}.
\end{romenumerate}
\end{definition}

\begin{definition}
  \label{D3}
A property of graphons $W$ is \emph{\qr} if every graphon $W$ that satisfies
it is \aex{} equal to a constant.
Furthermore, the property is \emph{\pqr} if it is satisfied only by graphons
$W$ that are \aex{} equal to $p$.
\end{definition}

We can now use standard arguments to translate our problem from graph
sequences to graphons. 
Recall that $m:=|F|$.

\begin{lemma}\label{Lequiv}
For any given graph $F$,  $p\in\oi$ and $\gamm\in(0,1)$ with $\sumim\ga_i\le1$, 
the  property $\cPl(F;\gamm)$ (of graph sequences)
is \pqr{} if and only if
the property $\cplw(F;\gamm)$ (of graphons) is.

Similarly, the property
$\cPu(F;\gamm)$ is \pqr{} if and only if  
the property
$\cpuw(F;\gamm)$ is, and
$\ctP(F;\gamm)$ is \pqr{} if and only if  
$\ctpw(F;\gamm)$ is.
\end{lemma}

\begin{proof}
Suppose that $\cP(F;\gamm)$ is \pqr, and let $W$ be a graphon satisfying 
$\cplw(F;\gamm)$. Let $\gnq$ be any sequence of graphs converging to $W$.
By assumption, \refL{LA1c}\ref{La1cii} holds with $\gam=p^{e(F)}$, and thus
\refL{LA1c} shows that \eqref{la1ci} holds for all disjoint $\uum$ with
$|U_i|=\floor{\ga_i|G_n|}$. In other words, $\gnq$ satisfies the property
$\cP(F;\gamm)$, and since this property was assumed to be \pqr, the sequence
$\gnq$ is \pqr, and thus $G_n\to W_p$, where $W_p=p$ everywhere. 
Since $G_n\to W$, this implies $W=W_p=p$
a.e.

Conversely, suppose that $\cplw(F;\gamm)$ is \pqr, and let $\gnq$ be a graph
sequence satisfying $\cPl(F;\gamm)$. This means that 
\refL{LA1c}\ref{La1ci} holds with $\gam=p^{e(F)}$.
Consider a subsequence of $\gnq$ that converges to some graphon $W$.
\refL{LA1c} then shows that \eqref{la1cii} holds for all disjoint $\aam$ with
$\gl(A_i)=\ga_i$. In other words, $W$ satisfies the property
$\cplw(F;\gamm)$, and since this property was assumed to be \pqr, 
$W=p$ a.e. Consequently, every convergent subsequence of $\gnq$ converges
to the constant graphon $W_p=p$. Since every subsequence has convergent
subsubsequences, it follows that the full sequence $\gnq$ converges to $W_p$,
\ie, $\gnq$ is \pqr.

The same proof works for $\ctP(F;\gamm)$ and $\ctpw(F;\gamm)$.
\end{proof}

In the rest of the paper we analyze the graphon properties 
$\cplw(F;\gamm)$ and $\ctpw(F;\gamm)$.

\section{The analytic part}\label{Sanal}

\citet{SJ234} proved the following lemma:

\begin{lemma}[{\cite[Lemma 7.3]{SJ234}}]
  \label{L3}
Let\/ $m\ge1$ and $\ga\in(0,1)$.
Suppose that 
$f$ is an integrable function on $\oi^m$
such that $\intaamx f=0$ for all sequences $\aam$ of
measurable subsets of\/ $\oi$ such that
$\gl(A_1)=\dots=\gl(A_m)=\ga$.
Then $f=0$ a.e.

Moreover, if $\ga<m\qw$, it is enough to consider
disjoint $\aam$.
\end{lemma}

It was remarked in \cite[Remark 7.4]{SJ234} that the second part (disjoint
subsets) 
of this lemma fails when $\ga=1/m$, \ie, when we consider partitions of
$\oi$ into $m$ disjoint sets of equal measure $1/m$
(we call these \emph{equipartitions});  
a simple counterexample is
provided by 
the following lemma.

\begin{lemma}\label{Lcounter}
Let $m\ge1$.
Suppose that 
\begin{equation}\label{fg}
  f(\xxm)=g(x_1)+\dots+g(x_m)
\end{equation}
for some integrable function $g$ on $\oi$ with $\intoi g=0$.
Then 
$f$ is a symmetric integrable function on $\oi^m$
and
\begin{equation}\label{sw0}
\intaamx f=0  
\end{equation}
for all partitions $\aamq$ of\/ $\oi$
into $m$ disjoint measurable subsets  such that
$\gl(A_1)=\dots=\gl(A_m)=1/m$.
\nopf
\end{lemma}

\begin{proof}
If $\aamq$ is an equipartition of $\oi$, then 
\begin{multline}
\label{fg0}
  \intaamx f(\xxm) =\sumim \Bigparfrac1m^{m-1}\int_{A_i} g(x_i)\dd x_i
\\
=m^{1-m}\sumim \int_{A_i} g(x)\dd x
=m^{1-m}\intoi g(x)\dd x=0.	
\end{multline}
\end{proof}

Moreover, it was shown in 
\cite[Proof of Lemma 9.4 and the comments after it]
{SJ234}, 
see also \cite[Lemma 10.3]{SJ234},
that if $m=2$ and $f$ is symmetric
with $\int_{A_1\times A_2}f=0$ for every equipartition $\set{A_1,A_2}$,
then $f$ has to be of the form \eqref{fg} a.e.
We shall here extend this to any $m$, thus showing that the converse to
\refL{Lcounter} holds.

\begin{lemma}\label{Lmain}
Let $m\ge1$.
Suppose that $f:\oi^m\to\bbC$ is a symmetric integrable function
such that 
\begin{equation}\label{sw}
\intaamx f=0  
\end{equation}
for all partitions $\aamq$ of\/ $\oi$
into $m$ disjoint measurable subsets  such that
$\gl(A_1)=\dots=\gl(A_m)=1/m$.
Then 
\begin{equation}\label{fgae}
  f(\xxm)=g(x_1)+\dots+g(x_m)
\qquad a.e.
\end{equation}
for some integrable function $g$ on $\oi$ with $\intoi g=0$.  
\end{lemma}

\begin{proof}
  The lemma is trivial when $m=1$. 
The case $m=2$ is, as said above,
  proved in \cite{SJ234}, but for completeness, we repeat the argument:

Let $f_1(x):=\intoi f(x,y)\dd y$. Then, for every subset $A\subset\oi$ with
$\gl(A)=1/2$, \eqref{sw} with  $A_1:=A$ and $A_2:=\oi\setminus A$ yields
\begin{equation}
  \begin{split}
0&=\int_{A_1\times A_2}f(x,y)\dd x\dd y 
= \int_{A\times\oi} f(x,y) \dd x\dd y-\int_{A\times A}f(x,y) \dd x\dd y
\\&
= \int_A f_1(x) \dd x-\int_{A\times A}f(x,y) \dd x\dd y
\\&
= \int_{A\times A}\bigpar{f_1(x)+f_1(y)-f(x,y)} \dd x\dd y.
  \end{split}
\end{equation}
The integrand in the last integral is symmetric, and it follows by 
\cite[Lemma  7.6]{SJ234} that it vanishes a.e., which proves \eqref{fgae}
with $g=f_1$; moreover, arguing as in \eqref{fg0}, for any equipartition
\set{A_1,A_2} of \oi,
\begin{equation}
  \begin{split}
0&=\int_{A_1\times A_2}f(x,y)\dd x\dd y 
= \frac12\intoi g(x)\dd x,
  \end{split}
\end{equation}
and thus $\intoi g = 0$, completing the proof when $m=2$.

Thus suppose in the remainder of the proof that $m\ge3$. 

\stepx{}\label{stepa}
Fix a subset $B\subset\oi$ with measure $\gl(B)=2/m$, and fix an
equipartition of the complement $\oi\setminus B$ into $m-2$ sets
$A_3,\dots,A_m$ of equal measure $1/m$.
Let 
\begin{equation}
f_2(x_1,x_2):=\int_{A_3\times\dots\times A_m}f(x_1,x_2,\dots,x_m)
\dd x_3\dotsm\dd x_m.  
\end{equation}
Then the assumption \eqref{sw} says that for any equipartition $B=A_1\cup
A_2$ of $B$ into two disjoint subsets of equal measure,
\begin{equation}
  \int_{A_1\times A_2}f_2(x_1,x_2)\dd x_1\dd x_2=0.
\end{equation}
The set $B$ is, as a measure space, isomorphic to $[0,2/m]$, and by a
trivial rescaling, the case $m=2$ shows that there exists an integrable
function $h$ on $B$ with $\int_B h=0$ such that
\begin{equation}
  f_2(x_1,x_2)=h(x_1)+h(x_2), \qquad\text{\aex{} }x_1,x_2\in B.
\end{equation}
This means that if $\psi_1$ and $\psi_2$ are bounded functions on $\oi$ such
that 
$\intoi\psi_1=\intoi\psi_2=0$ and $\supp(\psi_1)\cup\supp(\psi_2)\subseteq B$, 
then
  \begin{equation}\label{lars}
	\begin{split}
&
\int_{\oi^m} f(\xxm)\psi_1(x_1)\psi_2(x_2)
\etta_{A_3}(x_3)\dotsm \etta_{A_m}(x_m)
\dd x_1\dotsm\dd x_m
\\&\quad
=\int_{B\times B}f_2(x_1,x_2)\psi_1(x_1)\psi_2(x_2)\dd x_1\dd x_2	  
\\&\quad
=\int_B h(x_1)\psi_1(x_1)\dd x_1\int_B \psi_2(x_2)\dd x_2
+\int_B \psi_1(x_1)\dd x_1\int_B h(x_2)\psi_2(x_2)\dd x_2
\\&\quad
=0.
	\end{split}
\raisetag\baselineskip
  \end{equation}

\stepx{}\label{stepb}
Let us instead start with two bounded functions $\psi_1$ and $\psi_2$ on
$\oi$ such that 
$\intoi\psi_1=\intoi\psi_2=0$, and assume that
$\gl(\supp(\psi_1))+\gl(\supp(\psi_2))<2/m$.
Let $B_0:=\supp(\psi_1)\cup\supp(\psi_2)$ and $\boc:=\oi\setminus B_0$.
Then $\gl(B_0)<2/m$ and $\gl(\boc)=1-\gl(B_0)>(m-2)/m$.

Define 
\begin{equation}\label{f3}
  f_3(x_3,\dots,x_m):=\int_{\oi^2}
  f(x_1,x_2,\dots,x_m)\psi_1(x_1)\psi_2(x_2)\dd x_1\dd x_2.
\end{equation}
For any disjoint sets $A_3,\dots,A_m\subset \boc$ with
$\gl(A_3)=\dots=\gl(A_m)=1/m$, we can use Step \ref{stepa} with
$B:=\oi\setminus\bigcup_{3}^m A_i\supset B_0$ and conclude by \eqref{lars}
that
\begin{equation}\label{zam}
  \int_{A_3\times\dots\times A_m} f_3(x_3,\dots,x_m)
=
0.
\end{equation}
The set $\boc$ is, as a measure space up to a trivial rescaling of the measure,
isomorphic to $\oi$. Since $\gl(\boc)>(m-2)/m$, it follows by the second
part of \refL{L3} that \eqref{zam} (for arbitrary $A_3,\dots,A_m$ as above)
implies 
\begin{equation}\label{f0bc}
  f_3(x_3,\dots,x_m)=0, \qquad\text{a.e. }x_3,\dots,x_m\in\boc.
\end{equation}

\stepx{}\label{stepc}
Fix bounded functions $\gf_3,\dots,\gf_m$ on $\oi$.
For $B\subseteq\oi$, define
\begin{equation}\label{fb}
  \begin{split}
f_B(x_1,x_2):=
\int_{(B\comp)^{m-2}} f(x_1,\dots,x_m)\gf_3(x_3)\dotsm\gf_m(x_m).
  \end{split}
\end{equation}
If $\gl(B)>0$ and $\psi_1$ and $\psi_2$ are bounded functions with 
$\supp(\psi_\nu)\subseteq B$, $\gl(\supp(\psi_\nu))<1/m$ and
$\intoi\psi_\nu=0$, $\nu=1,2$, then Step \ref{stepb} shows, 
using  \eqref{fb},  \eqref{f3} and \eqref{f0bc}, 
since $B_0\subseteq B$ and thus $B\comp\subseteq\boc$, 
\begin{equation}\label{fb0}
  \begin{split}
\int_{\oi^2}f_B(x_1,x_2)\psi_1(x_1)\psi_2(x_2)
&=
\int_{(B\comp)^{m-2}} f_3(x_1,\dots,x_m)\gf_3(x_3)\dotsm\gf_m(x_m)\\
&=0.  	
  \end{split}
\raisetag{\baselineskip}
\end{equation}

Now suppose that $B$ is open, and $x_1,x_1',x_2,x_2'\in B$. 
For small enough $\eps>0$, the functions
\begin{equation}
\psi_\nu(x):=\frac{1}{2\eps}\bigpar{\etta_{(x_\nu-\eps,x_\nu+\eps)}(x)
-\etta_{(x'_\nu-\eps,x'_\nu+\eps)}(x)},
\qquad \nu=1,2,
\end{equation}
satisfy the conditions above and thus \eqref{fb0} holds.
Letting $\eps\to0$, it follows that if 
$(x_1,x_2), (x_1,x_2'), (x_1',x_2), (x_1',x_2')$
are Lebesgue points of $f_B$, then 
\begin{equation}\label{4point}
 f_B(x_1,x_2) -f_B(x_1,x_2') -f_B(x_1',x_2) +f_B(x_1',x_2')=0.
\end{equation}
Thus, \eqref{4point} holds for \aex{} $x_1,x_1',x_2,x_2'\in B$.

Consider now the countable collection $\cB$ of sets $B\subset(0,1)$ that are
unions of four open intervals with rational endpoints.
It follows that for \aex{} $x_1,x_1',x_2,x_2'\in \oi$,
\eqref{4point} holds for every set $B\in\cB$ such that
$x_1,x_1',x_2,x_2'\in B$.

Consider such a 4-tuple $x_1,x_1',x_2,x_2'\in \oi$. There exists a 
decreasing sequence $B_n$ of sets in $\cB$ with $\bigcup_1^\infty
B_n=\set{x_1,x_1',x_2,x_2'}$. Then \eqref{4point} holds for each $B_n$,
and by \eqref{fb} and dominated convergence,
$f_{B_n}(x,y)\to f_{\emptyset}(x,y)$ for all $x,y\in\oi$; hence,
\begin{equation}\label{4point4}
 f_\emptyset(x_1,x_2) -f_\emptyset(x_1,x_2') -f_\emptyset(x_1',x_2) +f_\emptyset(x_1',x_2')=0.
\end{equation}

\stepx\label{stepd}
Let $\gf_1,\dots,\gf_m$ be bounded functions on $\oi$
such that $\intoi\gf_1=\intoi\gf_2=0$.
Step \ref{stepc} shows that \eqref{4point4} holds 
for \aex{} $x_1,x_1',x_2,x_2'\in \oi$. 
We may thus fix $x_1',x_2'\in\oi$ such that \eqref{4point4} holds for \aex{}
$x_1,x_2$. 
Then multiply \eqref{4point4} by
$\gf_1(x_1)\gf_2(x_2)$ and integrate over $x_1,x_2\in\oi$. 
Since $\intoi\gf_1=\intoi\gf_2=0$, 
the integrals with the last three terms on the \lhs{} of \eqref{4point4}
vanish, and the result is
\begin{equation}
  \int_{\oi^2} f_\emptyset(x_1,x_2)\gf_1(x_1)\gf_2(x_2)\dd x_1\dd x_2=0.
\end{equation}
By the definition \eqref{fb}, this says
\begin{equation}\label{int0}
  \begin{split}
\int_{\oi^m}&f(x_1,\dots,x_m)\gf_1(x_1)\gf_2(x_2)\gf_3(x_3)\dotsm\gf_m(x_m)
=0.	
  \end{split}
\end{equation}

\stepx\label{stepe}
We may conclude in several ways. The perhaps simplest is to take 
$\gf_j(x)=e^{2\pi\ii n_jx_j}$, $j=1,\dots,m$ with $n_j\in\bbZ$ and
$n_1,n_2\neq0$. Step \ref{stepd} then applies and 
\eqref{int0} says that the Fourier coefficient
\begin{equation}
  \hf(n_1,\dots,n_m)=0
\end{equation}
when $n_1,n_2\neq0$.
Since $f$ is symmetric, it follows that $\hf(n_1,\dots,n_m)=0$ as soon as at
least two of the indices $n_1,\dots,n_m$ are non-zero.

Furthermore, let
\begin{equation}
  g(x_1):=\int_{\oi^{m-1}}f(\xxm)\dd x_2\dotsm \dd x_m-\int_{\oi^m} f,
\end{equation}
and note that $g$ is a function on $\oi$ with $\intoi g=0$, and let
\begin{equation}\label{h}
  h(x_1,\dots,x_m):=\sumim g(x_i)+\int_{\oi^m} f.
\end{equation}
Then 
$\hh(n_1,\dots,n_m)=0=\hf(n_1,\dots,n_m)$
as soon as at
least two of the indices $n_1,\dots,n_m$ are non-zero.
Moreover, when $n_1\neq0$,
\begin{equation}
  \widehat h(n_1,0,\dots,0)=\int_{\oi^m} h(x_1,\dots,x_n)e^{2\pi\ii n_1x_1} = 
\hg(n_1) =\hf(n_1,0,\dots,0)
\end{equation}
and thus by symmetry $\hh(n_1,\dots,n_m)=\hf(n_1,\dots,n_m)$ 
also when exactly one index $n_1,\dots,n_m$ is non-zero.
Finally, since
$\intoi g(x)\dd x=0$,
\begin{equation}
  \widehat h(0,\dots,0)=\int_{\oi^m} h = \int_{\oi^m} f =\hf(0,\dots,0).
\end{equation}
Consequently, $\hh(\nnm)=\hf(\nnm)$ for all $\nnm$ and thus $h=f$ a.e.

\stepx{}
We have shown that \aex{} $f=h$, given by \eqref{h}. 
Let $a:=\int f$; it remains to show that $a=0$. 
This is easy; using \eqref{h} and \refL{Lcounter},
  \begin{equation}
\intaamx f = 
\intaamx h 
=\intaamx a
= a\gl(A_1)\dotsm\gl(A_m),	
  \end{equation}
and thus the assumption \eqref{sw} yields $a=0$. 
\end{proof}

\begin{remark}\label{Rejsymm}
  As remarked in \citet[Remark 9.5]{SJ234}, it is essential that $f$ is
  symmetric in \refL{Lmain} (unlike \refL{L3}). 
For example, it is easily seen that 
the condition \eqref{sw} is also satisfied 
by every anti-symmetric $f$ such that the margin 
\begin{equation}\label{er}
\intoi f(x_1,\dots,x_m)\dd x_m=0  
\end{equation}
for \aex{} $x_1,\dots,x_{m-1}$;
in fact, \eqref{er} implies, for any partition \aamq, 
\begin{equation*}
  \begin{split}
\intaamx f
\dd x_1\dotsm \dd x_m
=-\sum_{k=1}^{m-1}\int_{A_1\times\dots\times A_{m-1}\times A_k} f
\dd x_1\dotsm \dd x_m
=0,
  \end{split}
\end{equation*}
since  each integral in the sum vanishes by the anti-symmetry.
As a concrete example, for any $m\ge2$, we may take the modified discriminant
\begin{equation}
  f(x_1,\dots,x_m) = e^{2\pi\ii (x_1+\dots+ x_m)}
\prod_{j<k}\bigpar{e^{2\pi\ii x_j}-e^{2\pi\ii x_k}}
\end{equation}
(or its real part).

For $m=2$, it is easy to see that every $f$ satisfying \eqref{sw} for all
equipartitions $\aamq$ is the sum of a symmetric function satisfying
\eqref{fgae} and an anti-symmetric function satisfying \eqref{er}, see
\cite[Remark 9.5]{SJ234}. For $m\ge3$, we do not know any characterization
of general $f$ satisfying \eqref{sw}, and we leave that as an open problem:
\end{remark}

\begin{problem}\label{Pasymm}
Find all integrable functions $f$ on $\oi^m$
(not necessarily symmetric) that satisfy
\eqref{sw}
for all partitions $\aamq$ of\/ $\oi$
into $m$ disjoint measurable subsets  such that
$\gl(A_1)=\dots=\gl(A_m)=1/m$.
\end{problem}

We end this section with another, much simpler, extension of 
\refL{L3}  to
$\gamm$ that may be different and possibly with $\sumim\ga_i<1$, 
with exception only of the exceptional case treated in \refL{Lmain} when
all $\ga_i$ are equal to $1/m$.
\begin{lemma}
  \label{L3+}
Let $m\ge1$ and let $\gamm\in(0,1)$ with $\sumim \ga_i\le1$.
Suppose that 
$f$ is an integrable function on $\oi^m$
such that 
\begin{equation}\label{swa}
\intaamx f=0  
\end{equation}
for all sequences $\aam$ of
disjoint measurable subsets of\/ $\oi$ such that
$\gl(A_i)=\ga_i$, $i=1,\dots,m$.
Suppose further that either
\begin{romenumerate}[-10pt]
\item $\sumim\ga_i<1$, or
\item $\sumim\ga_i=1$ but $(\gamm)\neq(1/m,\dots,1/m)$, and $f$ is symmetric.
\end{romenumerate}
Then $f=0$ a.e.
\end{lemma}

\begin{proof}
The case $m=1$ is included in \refL{L3}. (For $m=1$, (ii) cannot occur.)
The case (ii) with $m=2$ is  \cite[Lemma 9.4]{SJ234}.
The remaining cases are proved by induction (on $m$)
in the same way as the special
case in \cite[Lemma 7.3]{SJ234}; we sketch the proof and refer to \cite{SJ234}
for omitted details. 

We thus assume $m\ge2$, and in the case $m=2$
that (i) holds. Furthermore, if (ii) holds, we may assume
$\ga_m\neq\ga_{m-1}$ by permuting the coordinates.

We fix a set $A_1$ with $\gl(A_1)=\ga_1$ and consider the function
$$f_{A_1}(x_2,\dots,x_m):=\int_{A_1}f(x_1,\dots,x_m)\dd x_1
$$ 
on $B^{m-1}$
where $B:=\oi\setminus A_1$. $B$ is as a measure space isomorphic to $\oi$,
after rescaling the measure, and the hypothesis implies that $f_{A_1}$
satisfies a corresponding hypothesis on $B^{m-1}$; hence $f_{A_1}=0$ \aex{}
on $B^{m-1}$
by induction. It follows that \eqref{swa} holds for all disjoint sets $\aam$
with $\gl(A_1)=\ga_1$ and $\gl(A_2),\dots,\gl(A_m)$ arbitrary.

Now instead fix any disjoint $A_2,\dots,A_m$ with
$\sum_{i=2}^m\gl(A_i)<1-\ga_1$, 
and let $B':=\oi\setminus\bigcup_{i=2}^m A_i$.
Then
\eqref{swa} thus holds for any $A_1\subset B'$ with $\gl(A_1)=\ga_1$, and it
follows from the case $m=1$ applied to 
$\fAAm(x):=\int_{A_2\times\dots\times A_m}f(x,x_2,\dots,x_m)
 \dd x_2\dotsm\dd x_m$ that  $\fAAm(x)=0$ \aex; hence \eqref{swa} holds for
 all disjoint $A_1,\dots,A_m$ with $\sum_{i=2}^m\gl(A_i)<1-\ga_1$.
It follows that $f(x_1,\dots,x_m)=0$  for every Lebesgue point
$(x_1,\dots,x_m)$ of $f$ with $x_1,\dots,x_m$ distinct.
\end{proof}

\begin{remark}\label{RL3+}
  In \refL{L3+}(ii), the assumption that $f$ is symmetric is essential,
as is seen by the counterexample in \refR{Rejsymm}.
\end{remark}

\begin{remark}\label{RL3++}
The proof shows that in the case $\sumim\ga_i=1$, it suffices to assume that
$f$ is symmetric in the last two variables, provided $\ga_{m-1}\neq\ga_m$.
\end{remark}

We apply the results above 
to the property $\ctpw(F;\gamm)$. By \eqref{e34}, this property
says that \eqref{swa} holds for $f:=\tpsifw-p^{e(F)}$ and 
all disjoint
subsets $\aam$ of $\oi$ with $\gl(A_i)={\ga_i}$, $i=1,\dots,m$.

\begin{lemma}
  \label{LW}
Let $m\ge1$ and let $\gamm\in(0,1)$ with $\sumim \ga_i\le1$.
Suppose that\/ $W$ is a graphon and $p\in\oi$.
\begin{alphenumerate}[-10pt]
\item \label{lw1}
If\/ $(\gamm)\neq(1/m,\dots,1/m)$, 
then $\ctpw(F;\gamm)$ holds if and only if 
\begin{equation}
  \tpsifw(\xxm)=p^{e(F)}\quad\text{a.e.}
\end{equation}
\item \label{lw2}
If\/ $(\gamm)=(1/m,\dots,1/m)$, 
then $\ctpw(F;\gamm)$ holds if and only if 
there exists an integrable function $h$ with $\intoi h=p^{e(F)}/m$ such that
\begin{equation}\label{elw2}
  \tpsifw(\xxm)=\sumim h(x_i)\quad\text{a.e.}
\end{equation}
\end{alphenumerate}
\end{lemma}
\begin{proof}
  Part \ref{lw1} follows directly from \eqref{e34} and \refL{L3+}, 
while \ref{lw2}  follows from Lemmas \ref{Lcounter} and \ref{Lmain}, 
with $h(x)=g(x)+p^{e(F)}/m$.
\end{proof}

We thus see that the exceptional case $\ga_1=\dots=\ga_m=1/m$ in \refT{T=}
is more intricate than the cases covered by \refT{TA}.

We note also a similar result for $\cplw$. 
\begin{lemma}
 \label{LW3}
If\/ $\sumim\ga_i<1$,
then $\cplw(F;\gamm)$ holds if and only if 
\begin{equation}
  \psifw(\xxm)=p^{e(F)}\quad\text{a.e.}
\end{equation}
\end{lemma}
\begin{proof}
This too  follows from \eqref{e34} and \refL{L3+}.
\end{proof}
In this case we have to assume
$\sumim\ga_i<1$ for the proof, because
$\psifw$ is (in general) not symmetric, \cf{} Remarks \ref{RL3+} and
\ref{Rejsymm}. 

\begin{problem}\label{P4=1}
Does \refL{LW3} hold also if $\sumim\ga_i=1$ with 
$(\gamm)\allowbreak\neq(1/m,\dots,1/m)$?
\end{problem}

\section{Reduction to a two-type graphon}\label{Stwo}

We next reduce the problem by showing that, as for the similar problem
considered by 
\citet{SS:ind},
if the property $\ctpw(F;\gamm)$
is \emph{not} \qr, then there exists a counterexample with a 
2-type graphon.
This reduction reduces our problem to an algebraic one, which we consider in
the next section. 

We state the reduction in a somewhat general form, to be
used together with  \refL{LW}, and 
we give two versions (Theorems \ref{TD0} and \ref{TD}), to handle the two
cases in 
parts \ref{lw1} and \ref{lw2} in \refL{LW}.
The proofs are given later in this section.

\begin{remark}
\refT{TD0} is an extension of \citet[Theorem 5.5]{SJ234}, 
where $\Phi$ is a multiaffine polynomial,
which would be
sufficient for our application here. We nevertheless state \refT{TD0}
in order to show the similarities between Theorems \ref{TD0} and \ref{TD},
and because we now can give a more elegant proof of a more general statement
than in \cite{SJ234}, see \refR{RLD}.
\end{remark}

If  $\Phi\bigpar{(w\xij)_{i<j}}$ is a 
function of the  $\binom m2$ variables $w\xij$, $1\le  i<j\le m$, for some
$m\ge2$, and $W$ is a graphon,
we define, for $x_1,\dots,x_m\in\oi$,
\begin{equation}\label{phiw}
  \phiw(x_1,\dots,x_m):=\Phi\bigpar{(W(x_i,x_j))_{i<j}}.
\end{equation}

\begin{theorem}
  \label{TD0}
Suppose that $\Phi\bigpar{(w\xij)_{i<j}}$ is a 
continuous function of the
  $\binom m2$ variables $w\xij$, $1\le  i<j\le m$, for some $m\ge2$,
and let $a\in\bbR$.
Then the following are equivalent.
  \begin{romenumerate}[-10pt]
\item\label{td0w}
There exists a graphon\/ $W$ 
such that	
\begin{equation}\label{rang0}
  \phiw(\xxm)=a
\end{equation}
for \aex{} $\xxm\in\oi$, but\/ $W$ is \emph{not} \aex{} constant.

\item\label{td02}
There exists a $2$-type graphon\/ $W$ 
such that \eqref{rang0} holds 
for all $\xxm$, but\/ $W$ is
\emph{not} constant. 

\item\label{td0uvs}
There exist numbers $u,v,s\in\oi$, not all equal,  
such
that for every subset $A\subseteq[m]$, if we choose
\begin{equation}\label{uvs0}
  w\xij:=
  \begin{cases}
	u,&i,j\in A, \\
	v,&i,j\notin A, \\
	s,&i \in A,\,j\notin A \text{ or conversely},
  \end{cases}
\end{equation}
then 
\begin{equation}\label{rang0ab}
\Phi((w\xij)_{i<j})=a.
\end{equation}
  \end{romenumerate}
\end{theorem}

\begin{theorem}
  \label{TD}
Suppose that $\Phi\bigpar{(w\xij)_{i<j}}$ is a 
continuous function of the
  $\binom m2$ variables $w\xij$, $1\le  i<j\le m$, for some $m\ge2$.
Then the following are equivalent.
  \begin{romenumerate}[-10pt]
\item\label{tdw}
There exists a graphon\/ $W$ 
and a function $h$ on $\oi$, with $h$ not \aex{} $0$,
such that	
\begin{equation}\label{rang1}
  \phiw(\xxm)=\sumim h(x_i)
\end{equation}
for \aex{} $\xxm\in\oi$, but\/ $W$ is \emph{not} \aex{} constant.

\item\label{td2}
There exists a $2$-type graphon\/ $W$ 
and a function  $h$ on $\oi$, with $h$ not \aex{} $0$, 
such that \eqref{rang1} holds 
for all $\xxm$, but\/ $W$ is
\emph{not} constant. 

\item\label{tduvs}
There exist numbers $u,v,s\in\oi$, not all equal, and $a,b\in\bbR$, not both
$0$, such
that for every subset $A\subseteq[m]$, if we choose
\begin{equation}\label{uvs}
  w\xij:=
  \begin{cases}
	u,&i,j\in A, \\
	v,&i,j\notin A, \\
	s,&i \in A,\,j\notin A \text{ or conversely},
  \end{cases}
\end{equation}
then 
\begin{equation}\label{rang1ab}
\Phi((w\xij)_{i<j})=a+b|A|.  
\end{equation}
  \end{romenumerate}
\end{theorem}

\begin{remark}
In part \ref{td2} of Theorems \ref{TD0}--\ref{TD}, we may further require
that the two parts of 
$\oi$ are the intervals $[0,\frac12]$ and $(\frac12,1]$. 
Equivalently, we
may regard $W$ as a graphon defined on the two-point probability space
$(\setoi,\mu)$, with $\mu\set0=\mu\set1=\frac12$.
\end{remark}

\begin{remark}
  \refT{TD} holds also without the restrictions that $h$ is not \aex{} 0, and
  $a,b$ are not both 0; this follows by the same proof (with some
  simplifications).
Note that the excluded case, when $h=0$ \aex{}
  and $a=b=0$, is equivalent to \refT{TD0}.
For our purposes, it is
  essential that the case $a=b=0$ is excluded, since there are such
  examples that have to be excluded from our arguments,
for example the bipartite example in
\refR{RK3bi}, which corresponds to the case $u=v=0$, $s=1$ and
$\Phi((w\xij)_{i<j})=0$ for any $A$.
\end{remark}

The proofs follow the proof of \citet[Theorem 5.5]{SJ234}, with some
modifications. We prove the more complicated \refT{TD} in detail first, and
then sketch the similar but simpler proof of \refT{TD0}.

\begin{proof}[Proof of \refT{TD}]

\ref{td2}$\implies$\ref{tdw}: Trivial.

\ref{tduvs}$\implies$\ref{td2}: 
Define a 2-type graphon $W$ by
\begin{equation}\label{uvsW}
  W(x,y):=
  \begin{cases}
	u,& x,y > \frac12, \\
	v,& x,y \le \frac12, \\
	s,& x\le \frac12 <y \text{ or conversely},
  \end{cases}
\end{equation}
and let the function $h$ be
\begin{equation}\label{hex}
  h(x):=
  \begin{cases}
	a/m, & x\le \frac12,\\
a/m+b, & x>\frac12.
  \end{cases}
\end{equation}
Then 
\begin{equation}\label{vint}
\phiw(\xxm)=\Phi((w\xij)_{i<j})  
\end{equation}
where $w\xij$ is given by \eqref{uvs}
with $A:=\set{i:x_i>\frac12}$, and \eqref{rang1} follows from \eqref{rang1ab}.

\ref{tdw}$\implies$\ref{tduvs}: 
Suppose that $W$ is a graphon as in \ref{tdw},
but that \ref{tduvs} does not hold; we will show that this leads to a
contradiction. 
We first use \refL{LD} below, which (by replacing $W$ by $\bW$ and $h$ by
$\bh$) shows that we may assume that 
\eqref{rang1} holds for \emph{all}  $\xxm\in\oi$.

Suppose that $x,y\in \oi$. Given $A\subseteq[m]$,
let $x_i:=x$ for $i\in A$ and $x_i:= y$ for
$i\notin A$. Then $W(x_i,x_j)=w\xij$ as given by
\eqref{uvs} with $u=W(x,x)$, $v=W(y,y)$, $s=W(x,y)$.
Furthermore, \eqref{rang1} holds by our assumption, and thus
\begin{equation}\label{tom}
  \begin{split}
\Phi((w\xij)_{i<j})&=\phiw(\xxm)=\sumim h(x_i)= |A| h(x)+(m-|A|) h(y) 
\\ 
&= a+b|A| 	
  \end{split}
\raisetag{\baselineskip}
\end{equation}
with $a=mh(y)$ and $b=h(x)-h(y)$. Hence, \eqref{rang1ab} holds.
Since \ref{tduvs} does not hold, we must have either $u=v=s$ or
$a=b=0$. Note that $a=b=0$ if and only if $h(x)=h(y)=0$. 
Consequently, we have shown the following 
property:
\begin{equation}
  \label{md}
\text{If  $x,y\in \oi$, then } W(x,x)=W(y,y)=W(x,y)
\text{ or } h(x)=h(y)=0. 
\end{equation}

Furthermore, if $W(x,x)=W(y,y)$ then \eqref{tom}, with $A=\emptyset$ and
$A=[m]$, implies that
\begin{equation}
  a=\phiw(y,\dots,y)=\phiw(x,\dots,x)=a+mb
\end{equation}
and thus $b=0$ so $h(x)=h(y)$. Consequently, \eqref{md} implies that
\begin{equation}
  \label{hatt}
x,y\in \oi \implies h(x)=h(y).
\end{equation}
In other words, $h(x)=\gam$ for some constant $\gam\in\bbR$.

Note that $\gam\neq0$, since otherwise $h(x)$ would be 0 for all $x$,
contrary to the assumption \ref{tdw}. 
Hence, $h(x)\neq0$ for all $x$ and \eqref{md} implies
\begin{equation}
  \label{fd}
x,y\in\oi \implies W(x,x)=W(y,y)=W(x,y).
\end{equation}
Thus $W$ is  constant, contradicting the
assumption. 

This contradiction shows that \ref{tduvs} holds.
\end{proof}

\begin{proof}[Proof of \refT{TD0}]
  We argue as in the proof of \refT{TD}, with $b=0$ and $h(x)=a/m$;
in the proof of \ref{td0w}$\implies$\ref{td0uvs} we use \refL{LD0} instead
of \refL{LD}, and note directly that \eqref{tom} with $b=0$, which is
\eqref{rang0ab}, implies $u=v=s$ since \ref{tduvs} is assumed not to hold.
\end{proof}

\begin{remark}
 In both proofs, the proof of \ref{tduvs}$\implies$\ref{td2} also
works in the opposite direction and thus shows \ref{tduvs}$\iff$\ref{td2}
directly; \ref{tduvs} is just an explicit version of what \ref{td2} means.
\end{remark}

The proofs used the following technical lemmas, which both are
consequences of
a recent powerful general removal lemma by \citet{Petrov}.
Recall that a graphon $\bW$ is a version of $W$ if $W=\bW$ a.e.

\begin{lemma}\label{LD0}
  Suppose that $\Phi\bigpar{(w\xij)_{i<j}}$ is a continuous function of the
  $\binom m2$ variables $w\xij\in\oi$, $1\le  i<j\le m$, for some $m\ge2$.
Suppose further that\/ $W:\oi^2\to\oi$ is a graphon,
  \ie, a symmetric measurable function, and 
suppose that 
\begin{equation} \label{ld0}
 \phiw(\xxm)=a
\end{equation}
for some $a\in\bbR$ and \aex{} $\xxm\in\oi$.
Then there is a version\/ $\bW$ of\/ $W$
such that 
\begin{equation}\label{ld01}
\Phi_{\bW}(x_1,\dots,x_m)=a
\end{equation} 
for all
  $x_1,\dots,x_m\in\oi$.
\end{lemma}

\begin{proof}
  This is a direct application of 
\cite[Theorem 1(2)]{Petrov}, see \cite[Example 1]{Petrov}.
We let $M:=\Phi\qw(a) \subseteq \oi^{\binom m2}$
and note that \eqref{ld0} can be written
$\bigpar{W(x_i,x_j)}_{i<j}\in M$ for
\aex{} $\xxm$.
By \cite[Theorem 1(2)]{Petrov}, there exists a version $\bW$ of $W$ such
that $\bigpar{\bW(x_i,x_j)}_{i<j}\in M$ for
all $\xxm$, which is \eqref{ld01}.
(Petrov's theorem is stated for an infinite
sequence $x_1,x_2,\dots$, for maximal generality, but we can always ignore
all but any given finite number of the variables.)
\end{proof}

\begin{lemma}\label{LD}
  Suppose that $\Phi\bigpar{(w\xij)_{i<j}}$ is a continuous function of the
  $\binom m2$ variables $w\xij\in\oi$, $1\le  i<j\le m$, for some $m\ge2$.
Suppose further that\/ $W:\oi^2\to\oi$ is a graphon,
  \ie, a symmetric measurable function, and 
suppose that 
\begin{equation} \label{ld}
 \phiw(\xxm)=\sumim h(x_i)
\end{equation}
for some $h:\oi\to\bbR$ and \aex{} $\xxm\in\oi$.
Then there is a version\/ $\bW$ of\/ $W$
and a measurable function $\bh:\oi\to\bbR$ 
such that 
\begin{equation}\label{ld1}
\Phi_{\bW}(x_1,\dots,x_m)=\sumim \bh(x_i)
\end{equation} 
for all
  $x_1,\dots,x_m\in\oi$.
\end{lemma}

\begin{proof}
We translate \eqref{ld} into the setting of \cite{Petrov} as follows.

By \eqref{ld}, for \aex{} $x_1,\dots,x_m,y_1,\dots,y_m\in\oi$,
\begin{multline}
  \label{xy}
\phiw(\xxm)-\phiw(\yym)\\
=\sum_{\ell=1}^m
\bigpar{\phiw(x_\ell,y_1,\dots,\widehat{y_{\ell}},\dots,y_m)-
\phiw(y_1,\dots,\dots,y_m)},
\end{multline}
where $\widehat{y_{\ell}}$ means that this variable is omitted.
Let $x_{m+i}:=y_i$ ($1\le i\le m$) 
and $w_{ij}:=W(x_i,x_j)$ ($1\le i,j\le 2m$).
Then \eqref{xy} can be written as 
\begin{equation}\label{xyw}
  \Phix\Bigpar{(w_{ij})_{i\neq j}} =0
\end{equation}
for some continuous function $\Phix:\oi^{2m(2m-1)}\to\bbR$.
Let 
\begin{equation}
  M:=\Phix\qw(0)\subseteq\oi^{2m(2m-1)}.
\end{equation}
Since $\Phix$ is continuous, $M$ is a closed subset, and
by \eqref{xyw}, 
\begin{equation}
\bigpar{W(x_i,x_j)}_{i\neq j}\in M
\end{equation}
for \aex{} $x_1,\dots, x_{2m}$.
By \cite[Theorem 1(2)]{Petrov}, there exists a version $\bW$ of $W$
such that
\begin{equation}
\bigpar{\bW(x_i,x_j)}_{i\neq j}\in M
\end{equation}
for all $x_1,\dots, x_{2m}$. 
This means  that 
$\Phix\bigpar{(\bW(x_i,x_j))_{i\neq j}}=0$ for all $x_1,\dots,x_{2m}$, and thus
the analogue of \eqref{xy} for $\bW$ holds for all $\xxm,\yym$.
Now choose $y_1=\dots=y_m=0$. Then this analogue of \eqref{xy} yields
\eqref{ld1} with 
$\bh(x)= \Phi_{\bW}(x,0,\dots,0)-\frac{m-1}{m}\Phi_{\bW}(0,\dots,0)$.
\end{proof}

\begin{remark}\label{RLD}
\refL{LD0}, which follows from
Petrov's removal lemma \cite{Petrov},
is a simpler, stronger and more general version of 
\citet[Lemma 5.3]{SJ234}. 
Similarly, 
a modification of the proof of 
\cite[Lemma 5.3]{SJ234} can be used to prove
a weaker version of \refL{LD};
however, Petrov's  removal lemma 
enables us to a simpler and stronger statement with a simpler proof.
\end{remark}

\section{An algebraic condition}\label{Salg}

It is now easy to prove \refT{TA}. 

\label{Spf2}
\begin{proof}[Proof of \refT{TA}] 
\pfitemref{ta1}
  Suppose, in order to get a contradiction, that the property
  $\ctP(F;\gamm)$ is not   \pqr. By \refL{Lequiv}, also 
  $\ctpw(F;\gamm)$ is not   \pqr. 
That means that there exists a graphon $W$ that is not \aex{} equal to $p$
such that   $\ctpw(F;\gamm)$ holds, and thus by \refL{LW}\ref{lw1},
\begin{equation}\label{nyk}
  \tpsifw(\xxm)=p^{e(F)}\quad\text{a.e.}
\end{equation}
If $W$ \aex{} equals a constant, $w$ say, then $\tpsifw=w^{e(F)}$ a.e., and
thus $w^{e(F)}=p^{e(F)}$ and $w=p$, so $W=p$ \aex{} which we have
excluded. Hence, $W$ is not \aex{} constant.

Note that $\tpsifw(\xxm)$ by \eqref{psifw}--\eqref{symm} is a polynomial in
$W(x_i,x_j)$, $1\le i<j\le m$, 
and thus by \eqref{phiw} can be written as $\phiw(\xxm)$ for a 
suitable polynomial $\Phi$. 
We apply \refT{TD0}, with $a=p^{e(F)}$.
By \eqref{nyk}, \refT{TD0}\ref{td0w} holds, and thus \refT{TD0}\ref{td0uvs}
holds. 
Let $u,v,s$ be as there, and define $w\xij$ by \eqref{uvs0}.

Choosing $A=[m]$, we have $w\xij=u$ for all $i$ and $j$, and 
it is easily seen that $\Phi\bigpar{(w\xij)_{i<j}}=u\eef$
(see also \refL{Leipz} below); hence
\eqref{rang0ab} yields $u=p$. Similarly, the case $A=\emptyset$ yields
$v=p$.
Finally, take $A=\set1$, and regard $\Phi\bigpar{(w\xij)_{i<j}}$ as a polynomial
in $s$. Since $u,v>0$ and $e(F)>0$, this polynomial has non-negative
coefficients and at least one non-zero term with a positive power of $s$;
hence the polynomial is strictly increasing in $s>0$, so 
\eqref{rang0ab} has at most one root $s$. 
However, when $u=v=p$,
\eqref{rang0ab} is satisfied by $s=p$, and thus this is the
only root. 
Consequently, $u=v=s=p$, a contradiction, which completes the proof.

\pfitemref{ta2}
Similar, using Lemmas \ref{Lequiv} and \ref{LW3} and \refT{TD0}.
\end{proof}

\begin{remark}\label{RP3b}
  Suppose that the graph $F$ contains two vertices that are \emph{twins},
  i.e., such that the map interchanging these vertices and fixing all others
  is an automorphism. Label $F$ such that the twins are vertices $m-1$ and
  $m$.
The argument in the proof of \refT{TA} shows, using \refR{RL3++}, that
$\cPl(F;\gamm)$ is \qr{} provided $\ga_{m-1}\neq\ga_m$. (We do not know
whether this extends to $\ga_m=\ga_{m-1}$. Cf.\ \refP{P4=1}.)
In particular, this applies to $F=P_3$, see \refE{EP3} and \refP{PP3}.
\end{remark}

For \refT{T=}, the algebra is more complicated, and we analyse the
condition \eqref{rang1ab} as follows.

For a subset $A$ of $V(F)$, let $e_F(A)$ be the number of edges in $F$ with
both endpoints in $A$; similarly, if $A$ and $B$ are disjoint subsets of
$V(F)$, let $e_F(A,B)$ be the number of edges with one endpoint in $A$ and
the other in $B$. 
Further, let $A\comp:=V(F)\setminus A$ be the complement of $A$. 

\begin{lemma}\label{Leipz}
 Suppose that $F$ is a graph with\/ $|F|=m$ and let\/ $W$ be the $2$-type
graphon given by \eqref{uvsW} for some $u,v,s\in\oi$. 
Let $\xxm\in\oi$ and let $k:=|\set{i\le m:x_i>1/2}|$. Then
\begin{equation}
\tpsifw(\xxm)=
\binom mk\qw 
\sum_{A\subseteq V(F):|A|=k} u^{e_F(A)}v^{e_F(A\comp)}s^{e_F(A,A\comp)}.
\end{equation}
\end{lemma}

\begin{proof}
  Let $A:=\set{i\le m:x_i>1/2}$. 
Then by \eqref{psifw} and \eqref{uvsW},
\begin{equation}
\psifw(\xxm)=  u^{e_F(A)}v^{e_F(A\comp)}s^{e_F(A,A\comp)}. 
\end{equation}
By \eqref{symm}, $\tpsifw(\xxm)$ is the average of this over all
permutations of $\xxm$, which means taking the average over the $\binom mk$
sets $A\subseteq[m]$ with $|A|=k$.
\end{proof}

\begin{lemma}\label{Lalg}
Suppose that $F$ is a graph with 
$|F|=m$.
Then the following are equivalent.
\begin{romenumerate}[-10pt]
\item \label{lalg0}
For some $p\in(0,1]$,
$\ctP(F;\mmm)$ is \emph{not} \pqr.
\item \label{lalgw}
For some $p\in(0,1]$,
$\ctpw(F;\mmm)$ is \emph{not} \pqr.
\item \label{lalguvs}
There exist numbers $u,v,s\ge0$, not all equal, 
and some real $a$ and $b$, not both $0$, such that
\begin{equation}\label{alg}
  \sum_{A\subseteq V(F):|A|=k} u^{e_F(A)}v^{e_F(A\comp)}s^{e_F(A,A\comp)}
=\binom mk \xpar{a+bk},
\qquad k=0,\dots,m.
\end{equation}
\item \label{lalgq}
There exist numbers $u,v,s\ge0$, not all equal, 
such that the polynomial \uppar{in $q$}
\begin{equation}\label{algq}
\gluvs(q):=
  \sum_{A\subseteq V(F)} 
u^{e_F(A)}v^{e_F(A\comp)}s^{e_F(A,A\comp)}
q^{|A|}(1-q)^{m-|A|}
\end{equation}
has degree at most $1$, but does not vanish identically.
\item \label{lalgx} 
There exist numbers $u,v,s\ge0$, not all equal,
such that the polynomial \uppar{in $x$}
\begin{equation}\label{algx}
\glxuvs(x):=
  \sum_{A\subseteq V(F)} 
u^{e_F(A)}v^{e_F(A\comp)}s^{e_F(A,A\comp)}
(x-1)^{|A|}
\end{equation}
is divisible by $x^{m-1}$, but does not vanish identically.
\end{romenumerate}
\end{lemma}

Note that (for $q\in\oi$)
$\gluvs(q)$ is the expectation of
$u^{e_F(A)}v^{e_F(A\comp)}s^{e_F(A,A\comp)}$ if $A$ is the random subset 
$[m]_q$ of
$[m]$ obtain by including each element with probability $q$, independently
of each other.

\begin{proof}
  \ref{lalg0}$\iff$\ref{lalgw}:
This is contained in \refL{Lequiv}.

\ref{lalgw}$\implies$\ref{lalguvs}:
If \ref{lalgw} holds, then there exists a graphon $W$ that is not \aex{}
constant for which $\ctpw(F;\mmm)$ holds. Then, by \refL{LW}\ref{lw2}, there
exists an integrable function $h$ with $\intoi h\neq 0$ such that 
\eqref{elw2} holds. 

As in the proof of \refT{TA}, 
$\tpsifw(\xxm)$ 
can be written as $\phiw(\xxm)$ for a 
polynomial $\Phi$. 
Then \eqref{elw2} is the same as \eqref{rang1} and \refT{TD}\ref{tdw} holds.
By \refT{TD} (and its proof) we may assume that $W$ is a 2-type graphon
given by \eqref{uvsW} for some $u,v,s\in\oi$, and then \refL{Leipz} and
\eqref{vint} show 
that
\eqref{rang1ab} is equivalent to \eqref{alg}, and thus \ref{lalguvs} follows.

\ref{lalguvs}$\implies$\ref{lalgw}:
This is similar but simpler.
We may assume that $u,v,s\in\oi$, by multiplying them by a small positive
number if necessary. Let $W$ be the graphon defined by \eqref{uvsW}. Then
\refL{Leipz} and \eqref{alg} yield $\tpsifw(\xxm)=a+bk$ where
$k=|\set{i:x_i>1/2}$, so \eqref{elw2} holds with $h$ given by \eqref{hex}. 

We have assumed that $a$ and $b$ are not both 0, and thus  $h(x)$ is not
identically 0. Furthermore, \eqref{elw2} implies $h(x)\ge0$ a.e., and thus
$\intoi h>0$. Since $\tpsifw\le1$, \eqref{elw2} also implies $\intoi h\le
1/m$. Hence there exists $p\in(0,1]$ with $p\eef=m\intoi h$.
(Also in the trivial case $e(F)=0$, since then $\tpsifw=1$.) 
\refL{LW} now
shows that $\ctpw(F;\mmm)$ holds, and since $W$ is not \aex{} constant, this
yields \ref{lalgw}.

\ref{lalguvs}$\iff$\ref{lalgq}:
By multiplying \eqref{alg} by $t^k$ and summing over $k$, we see that
\eqref{alg} is equivalent to
\begin{equation}\label{algt}
  \sum_{A\subseteq V(F)} u^{e_F(A)}v^{e_F(A\comp)}s^{e_F(A,A\comp)} t^{|A|}
=\sum_{k=0}^m\binom mk \xpar{a+bk}t^k,
\qquad t\in\bbR.
\end{equation}
Letting $t=q/(1-q)$ and multiplying by $(1-q)^m$, 
this is equivalent to
\begin{equation*}
\sum_{A\subseteq V(F)} u^{e_F(A)}v^{e_F(A\comp)}s^{e_F(A,A\comp)} 
q^{|A|}(1-q)^{m-|A|}
=\sum_{k=0}^m\binom mk \xpar{a+bk}q^k(1-q)^{m-k}
\end{equation*}
where the right hand side equals $a+bmq$ by an elementary calculation (or by
the formula for the mean of a binomial distribution).
The equivalence follows.

\ref{lalgq}$\iff$\ref{lalgx}: 
Take $q=1/x$, replace $A$ by $A\comp$ and interchange $u$ and $v$ to obtain
\begin{equation}
  \glxuvs(x) = x^m \gL_{F;v,u,s}(1/x).
\end{equation}
\end{proof}

\begin{remark}\label{R1}
  It follows from the proof that the polynomial $\gluvs(q)$ has degree 0,
\ie, is a (non-zero) constant
$\iff$ \eqref{alg} holds with $b=0$
$\iff$ $\tpsifw(\xxm)=a$ for some (non-zero) $a$.
As shown above in the proof of \refT{TA}, this happens 
for some $u,v,s\ge0$, not all equal,
only in the trivial
case $e(F)=0$. 
(This is an equivalent way of stating the algebraic part of the proof of
\refT{TA}, but we preferred to give a direct proof above without the present
machinery.)
Hence, if $e(F)>0$ and \ref{lalgq} holds, then the degree of
$\gluvs$ is exactly 1.
\end{remark}

\begin{remark}\label{Risola}
$\gluvs(q)$ is not changed if we add some isolated
vertices to $F$. Hence we may assume that $F$ has no isolated vertices.
\end{remark}

We note that the cases $k=0$ and $k=m$ of \eqref{alg} simply are
\begin{align}
  v^{e(F)}&=a, \label{alg0}
\\ 
  u^{e(F)}&=a+mb. \label{algm}
\end{align}
In particular, the assumption that not $a=b=0$ means that not  $u=v=0$.
(This case has to be excluded, for any non-bipartite $F$, \cf{} \refR{RK3bi}.)

Moreover, if $F$ has degree sequence $d_1,\dots,d_m$, the cases $k=1$ and
$k=m-1$ of \eqref{alg} are
\begin{align}
\frac{1}m\sumim  v^{e(F)-d_i}s^{d_i}&=a+b, \label{alg1}
\\ 
\frac{1}m\sumim  u^{e(F)-d_i}s^{d_i}&=a+(m-1)b. \label{algm-1}
\end{align}

\begin{example}\label{EK2}
If $F=K_2$, then by \eqref{algq},
\begin{equation*}
  \gluvs(q)=uq^2+2sq(1-q)+v(1-q)^2=v+2(s-v)q+(u+v-2s)q^2,
\end{equation*}
which has degree 1 if we choose any distinct $u$ and $v$ and let $s=(u+v)/2$.
Hence \refL{Lalg} shows that
$\ctP(K_2;1/2,1/2)$ is not \qr, as we already know, see \refE{ERK2}.

In this case, $\Psi_{K_2,W}(x_1,x_2)=W(x_1,x_2)$, so
\refL{LW}\ref{lw2}  shows that
$\ctpw(K_2;1/2,1/2)$ holds if and only if $W(x,y)=h(x)+h(y)$ for some
measurable $h:\oi\to\oi$ with $\intoi h=p/2$, see further
\citet[Section 9]{SJ234}.
\end{example}

\begin{remark}\label{R2}
We may add some further conditions on $u,v,s$ in
\refL{Lalg}\ref{lalguvs}--\ref{lalgx}. 
In the trivial case $e(F)=0$ we can take any $u,v,s$, so let us assume
$e(F)>0$. By \refR{R1}, we then must have $b\neq0$, so
by \eqref{alg0}--\eqref{algm}, $u\neq v$.
Furthermore, 
we may interchange $u$ and $v$ (and replace $q$ by $1-q$ in \eqref{algq}),
so we may assume $u<v$. In this case, \eqref{alg0}--\eqref{algm} yield $b>0$.
By \eqref{alg1} and \eqref{alg0}, this implies $s>v$, and
by \eqref{algm-1} and \eqref{algm}, it implies $s<u$.
Hence we may assume $v<s<u$.


Suppose $v=0$. Then $a=0$ by \eqref{alg0}. 
By  \refR{Risola}, we may  assume that $F$ has no isolated vertices. 
If $d_i<e(F)$ for all $i$, then
\eqref{alg1} yields $0=a+b=b$, which is impossible. Hence we must have
$d_i=e(F)$ for some $i$, which means that $F$ is a star.
In the case of a star with $m=|F|\ge3$, $v=a=0$ in \eqref{alg1} yields
$s^{m-1}=mb$, while 
\eqref{algm} yields $u^{m-1}=mb$ so $u=s$, a contradiction.
Hence $v=0$ is impossible and we may assume $v>0$.
(If $m=2$, so  $F=K_2$, $v=0$ is possible, but we may choose any $v>0$ and
$u>v$ by \refE{EK2}.) 

Consequently, it suffices to consider distinct $u,v,s>0$, and we may
assume $0<v<s<u$ (or, by symmetry, $0<u<s<v$).

Furthermore, the equations \eqref{alg} are homogeneous in $(u,v,s)$, so we
may assume that any given of them equals 1; for example, we may
assume  $v=1$, which implies $a=1$ by \eqref{alg0}.
\end{remark}

\section{Completing the proof of \refT{T=}}\label{Spf1}

We say that a graph $F$ is \emph{good} if, for every $p\in(0,1]$, 
$\ctP(F;\mmm)$ is \pqr; otherwise $F$ is \emph{bad}.
In this terminology, \refL{Lalg} says 
(using \refR{R2})
that $F$ is bad if and only if there
exist distinct $u,v,s>0$  such
that
\eqref{alg} holds, or, equivalently, that $\gluvs(q)$ 
in \eqref{algq} has degree at most 1.

An empty graph, i.e., a graph $F$ with $e(F)=0$,
is trivially bad; in this case \eqref{algq} yields $\gluvs(q)=1$, so
$\gluvs$ has 
degree 0. By \refR{R1}, this is the only case when $\deg(\gluvs)=0$.

The single edge $K_2$ is also bad, see Examples \ref{ERK2} and \ref{EK2}. 
More generally, any graph $F$ with $e(F)=1$ is bad by \refR{Risola}.

\refConj{Conj1} says that all other graphs are good.
We proceed to verify this in the cases given in \refT{T=}.

\begin{example}[regular graphs]\label{Ereg}
Suppose that $F$ is $d$-regular for some $d\ge1$, and that $m=|F|\ge3$.
(This includes the case $K_m$, $m\ge3$, considered by \cite{HuangLee}.)
Then $e(F)=dm/2$.

We use only \eqref{alg0}--\eqref{algm-1}; if we further simplify by assuming
$v=a=1$, as we may by \refR{R2}, we obtain, from \eqref{algm}--\eqref{algm-1},
\begin{align}
  u^{d m/2} &= 1+mb, \label{twm}
\\
 s^{d} &=1+b, \label{tw1}
\\
u^{d(m-2)/2} s^{d} &=1+(m-1)b,\label{twm-1}
\end{align}
and thus
\begin{equation}\label{bx}
  (1+(m-1)b)^m = (1+mb)^{m-2} (1+b)^{m}.
\end{equation}
However, the function 
\begin{equation}
h(x):=  (m-2)\log(1+mx) + m\log(1+x) - m\log(1+(m-1)x)
\end{equation}
(defined for $x>-1/m$)
has derivative 
\begin{equation}
  h'(x) = \frac{m(m-1)(m-2)x^2}{(1+x)(1+(m-1)x)(1+mx)} >0
\end{equation}
for $x>-1/m$ with $x\neq 0$, and thus $h(x)$ is strictly increasing on
$(-1/m,\infty)$ and $h(x)\neq h(0)=0$ for $x\neq 0$, which shows that
\eqref{bx} implies $b=0$, and thus $s=u=1=v$ by \eqref{twm}--\eqref{twm-1},
a contradiction. Consequently, there are no $u,v,s$ satisfying the
conditions and thus $F$ is good.
\end{example}

\begin{example}[stars] \label{Estar}
Suppose that $F$ is a star $K_{1,m-1}$. 
Let $A\subseteq V(F)$ and let $k:=|A|$.
If $A$ contains the centre of $F$, then 
$e_F(A)=k-1$, $e_F(A\comp)=0$ and $e_F(A,A\comp)=m-k$; 
otherwise,
$e_F(A)=0$, $e_F(A\comp)=m-k-1$ and $e_F(A,A\comp)=k$.
It follows from \eqref{algx} and the binomial theorem that
\begin{equation}\label{q0}
  \glxuvs(x)=
(x-1)\bigpar{u(x-1)+s}^{m-1} + \bigpar{s(x-1)+v}^{m-1}.
\end{equation}

Assume $m\ge3$, and that $F$ is bad. Then, by \refL{Lalg}\ref{lalgx} and
\refR{R2}, there exist distinct $u,v,s>0$ 
such that $\glxuvs(x)$ is divisible by $x^{m-1}$.
In particular,
\begin{equation}\label{qa}
0=  \glxuvs(0)=- \xpar{s-u}^{m-1}+(v-s)^{m-1}.
\end{equation}
Hence
$(s-u)^{m-1} =(v-s)^{m-1}$ and thus 
$|s-u|=|v-s|$, and since $u,v,s$ are real, $s-u=\pm(v-s)$. 
However, we assume $u\neq v$ and thus $s-u\neq s-v$. Consequently,
$s-u=v-s$.

We may further assume $s=1$, and thus $u=1-y$ and $v=1+y$ for some 
$y\neq 0$. Thus, by \eqref{q0}, 
\begin{equation}\label{q3}
  \glxuvs(x)=
(x-1)\bigpar{(1-y)x+y}^{m-1} + \bigpar{x+y}^{m-1}.
\end{equation}
Since $m\ge3$, $\glxuvs(x)$ is divisible by $x^2$, so the derivative
$\glxuvs'(0)=0$. Hence, 
\begin{equation}
  \begin{split}
0=\glxuvs'(0)&=
y^{m-1}-(m-1)(1-y)y^{m-2}+(m-1)y^{m-2}
\\&
=my^{m-1}\neq0.
  \end{split}
\end{equation}
This is a contradiction, which shows that $F=K_{1,m-1}$ is good when $m\ge
3$.
(For $m=2$, $K_{1,1}=K_2$ is bad, as remarked above.)
\end{example}

\begin{example}[disconnected graphs]\label{Edis}
  Suppose that $F=\bigcup_{i=1}^k F_i$ is disconnected with components
  $F_1,\dots,F_k$. It follows easily from \eqref{algq} that then
  \begin{equation}
\gL_{F;u,v,s}(q)
= \prod_{i=1}^k \gL_{F_i;u,v,s}(q).
  \end{equation}
A component $F_i$ with $|F_i|=1$ has 
$\gL_{F_i;u,v,s}(q)=1$ and can be ignored, as said in \refR{Risola}.
On the other hand, if $|F_i|\ge2$, and thus $e(F_i)>0$, then by \refR{R1},
$\gL_{F_i;u,v,s}(q)$ has degree at least 1 whenever $u,v,s\ge0$ 
are not all equal.
Consequently, if there are at least 2 components with more than one vertex,
then $\gluvs(q)$ has degree at least 2, and thus $F$ is good.
\end{example}

This ends our (short) list of classes of graphs that are known to be good,
and completes the proof of \refT{T=}.
We can give further examples of individual small good graphs $F$ as follows.

\begin{example}[computer algebra]\label{Edd}
Fix a graph $F$ and
consider again the four equations \eqref{alg0}--\eqref{algm-1}.
If we set $s=1$ (see \refR{R2}), we can eliminate $a$ and $b$ and obtain the
two equations
\begin{align}
  \sumim  u^{e(F)-d_i}&=(m-1)u^{e(F)}+v^{e(F)}, \label{ddu}
\\  
\sumim  v^{e(F)-d_i}&=u^{e(F)}+(m-1)v^{e(F)} \label{ddv}
.
\end{align}
Since these are two  polynomial equations in two unknowns, there are
plenty of complex solutions $(u,v)$. However, if $F$ is bad, then by
\refL{Lalg} and \refR{R2} there exists a solution with $0<u<1<v$,
and by symmetry another solution with $0<v<1<u$.
Using computer algebra (in our case \maple), we can check this by writing
\eqref{ddu}--\eqref{ddv} as $f_1(u,v)=0$ and $f_2(u,v)=0$
and then computing the resultant $R(u)$ of $f_1(u,v)$ and $f_2(u,v)$ as
polynomials in $v$. Then the roots of $R(u)$ are exactly the values $u$ such
that \eqref{ddu}--\eqref{ddv} have a solution $(u,v)$ for some $v$.
Hence, if $F$ is bad, then $R(u)$ has at least one root in 
the interval $(0,1)$ and at least one root in $(1,\infty)$.
Consequently, if we compute the number of roots of $R(u)$
in $(0,1)$ and  in $(1,\infty)$
(by Sturm's theorem, this can be done using exact integer arithmetic),
and one of these numbers is $0$, then $F$ is good.

In general, this is perhaps too much to hope for. But even if there are
such roots, we can proceed by 
calculating the roots numerically. 
If the roots of $R(u)$ in  $(0,1)$ are $u_1,\dots,u_p$ and the roots in
$(1,\infty)$ are $v_1,\dots,v_q$,
then a solution of \eqref{ddu}--\eqref{ddv} with $0<u<1<v$ has to be one of
$(u_i,v_j)$; 
hence, if we check the pairs $(u_i,v_j)$ one by one and find that
none satisfies both \eqref{ddu} and \eqref{ddv}, then $F$ is good.
(Assuming that the computer calculations
are done with enough accuracy. It might be possible to find an algorithm
using exact arithmetic to 
test whether \eqref{ddu} and \eqref{ddv} have a common solution in
$(0,1)\times(1,\infty)$, but we have not investigated that.) 


We give some explicit examples where this method succeeds.
\end{example}

\begin{example}[paths] \label{Epath}
The path $P_2=K_2$ is bad, and the path $P_3=K_{1,2}$ is good by \refE{Estar}.  
For $F=P_4$ we have $m=4$, $e(F)=3$ and the degree sequence $1,2,2,1$.
The equations \eqref{ddu}--\eqref{ddv} are
$2u+2u^2=3u^3+v^3$ and $2v+2v^2=u^3+3v^3$, and the resultant 
$R(u)=
-512\,{u}^{9}+1152\,{u}^{8}+288\,{u}^{7}-1872\,{u}^{6}+288\,{u}^{5}+
976\,{u}^{4}-112\,{u}^{3}-192\,{u}^{2}-16\,u.
$
In this case, $R(u)$ has no roots in $(0,1)$, so $P_4$ is good.

For $P_5$, the resolvent $R(u)$ (now of degree 16) has a single root in
$(0,1)$, but no root in $(1,\infty)$, so $P_5$ is good.
(As an illustration, the root in $(0,1)$ is
$u=0.23467\dots$; for this root, \eqref{ddu} and \eqref{ddv}
have a common root $v=-0.65039\dots$, but no common root in $(1,\infty)$.)

We have investigated $P_m$ for  $4\le m\le20$, and the same pattern holds:
For even $m$,  the resolvent has no root in $(0,1)$ (but one root in
$(0,\infty)$). 
For odd $m$,  the resolvent has no root in $(1,\infty)$ (but one root in
$(0,1)$).
In both cases, $P_m$ is good. 

We conjecture that this pattern holds for all $m\ge4$.
\end{example}

\begin{example}[Graphs of size $|F|=4$]
Of the 9 graphs with $m=|F|=4$ and $e(F)>1$, 3 are disconnected (\refE{Edis}),
2 more are regular (\refE{Ereg}), 1 is a star (\refE{Estar}) and 1 is a path
(\refE{Epath}). The two remaining ones have degree sequences 
$(3,2,2,1)$ and $(3,3,2,2)$. In both cases, the resolvent $R(u)$ has no root
in $(0,1)$. Thus every $F$ with $|F|=4$ and $e(F)>1$  is good.
\end{example}

\begin{example}[complete bipartite graphs]\label{EKmn}
We have used the method in \refE{Edd} to verify that the
complete bipartite graphs
$K_{2,n}$ ($n\le 8$),  
$K_{3,n}$ ($n\le 7$),  
$K_{4,n}$ ($n\le 5$)
are good. In all cases, the resolvent $R(u)$ lacks roots in either $(0,1)$
or $(1,\infty)$, and sometimes in both. 
(For example, for $K_{2,n}$, there is no root in
$(1,\infty)$ for any $n\le8$, and a root in $(0,1)$ only for $n=4$ and
$n=8$. It is not clear whether this extends to larger $n$.)
\end{example}

\begin{remark}
We have so far not found any example with $e(F)>1$ where the method in
  \refE{Edd}  fails. 
We thus guess that if $e(F)>1$, then \eqref{ddu}--\eqref{ddv} have
  no common root with $0<u<1$ and $1<v<\infty$. (Equivalently,
  \eqref{alg0}--\eqref{alg1} have no common root with $0<u<s<v$.)
However, note that even if there is a graph $F$ for which this fails, $F$
still may be good since, if $m>3$, there are $m-3$ more equations
\eqref{alg} that have to be satisfied, which seems very unlikely.
In Examples \ref{Edd}--\ref{EKmn} we consider only the equations that only
depend on the degree sequence. 
\end{remark}

\section{More parts than vertices}\label{Smore}

\citet{ShapiraYuster:hyper} and
\citet{HuangLee} 
considered also (for $F=K_m$) the case of a partition
$\uur$ of $V(G_n)$ with $r>m$, where they count the number of copies of
$K_m$ with at most one vertex in each part $U_i$.

We can extend this to arbitrary graphs $F$ 
(as in \cite[Question 5.1]{HuangLee}).
In our notation this is the same as considering
(counting labelled copies and dividing by $m!$, where $m=|F|$)
\begin{equation}\label{nr}
  \sumiir \tN(F,G;\xim U),
\end{equation}
and we define the property $\ctP(F;\garr)$ 
for a sequence $(G_n)$ to mean
\begin{equation}\label{pr}
  \sumiir \tN(F,G_n;\xim U) = p\eef \sumiir\prodjm |U_{i_j}| +o\bigpar{|G_n|)^m}
\end{equation}
for all disjoint subsets $\uur$ of $V(G_n)$ with $|U_i|=\floor{\ga_i|G_n|}$,
$1\le i\le r$.
(For $r=m$, this yields the same property as before.)

In the case $0<p<1$, $r\ge m\ge3$, $F=K_m$ and $\sum_{i=1}^r\ga_i=1$.
\citet{ShapiraYuster:hyper} ($(\garr)\neq(1/r,\dots,1/r))$ and
\citet{HuangLee} ($(\garr)=(1/r,\dots,1/r))$ showed that this property is \pqr.
We can extend this as follows.

\begin{theorem}\label{Tr}
Let $F$ be a graph with $e(F)>0$, 
and let $0<p\le1$.
Further, let  $(\garr)$ be a vector of positive numbers of length 
$r\ge m=\ff$ with
$\sumir\ga_i\le1$.
If either  $(\garr)\neq(1/r,\dots,1/r)$
or $F$ is as in \refT{T=}, then $\ctP(F;\garr)$ is a \pqr{} property.
\end{theorem}

\begin{proof}
  The case $(\garr)=(1/r,\dots,1/r)$ is simple; in this case (and more
  generally when all $\ga_i$ are equal), it is easy to see that 
$\ctP(F;\garr)$ is the same as $\ctP(F_*;\garr)$, where $F_*$ is the graph
  with $r$ vertices obtained by adjoining $r-m$ isolated vertices to $F$; by
  \refL{Lalg} and \refR{Risola}, this property is \pqr{} if and only if 
$\ctP(F;\mmm)$ is, so the result in this case is equivalent to \refT{T=}.

In general, we note first that Lemmas \ref{LA1c} and \ref{Lequiv} extend 
(with the same proofs) and show that it is equivalent to consider the
property of graphons
\begin{equation}\label{rpsi}
\sumii \int_{A_{i_1}\times\dots\times A_{i_m}} \tpsiqfw(\xim x)
=p\eef\sumii\prodjm \gl(A_{i_j})    
\end{equation}
for all disjoint subsets $\aar$ of $\oi$ with $\gl(A_i)=\ga_i$.

Assume this and define
\begin{equation}\label{tpsiqq}
  \tpsiqqfw(\xxr):=\sumiir \tpsiqfw(\xim x)
\prod_{k\notin\set{i_1,\dots,i_m}}\ga_k\qw.
\end{equation}
Then \eqref{rpsi} can be written
\begin{equation}\label{rpsix}
 \int_{A_{1}\times\dots\times A_{r}} \tpsiqqfw(\xxr)
=p\eef\sumii\prodjm \ga_{i_j}    
\end{equation}
for all such subsets $\aar$.
Suppose now $(\garr)\neq(1/r,\dots,1/r)$. Then \refL{L3+} applies 
(to $\tpsiqqfw-\gam$ for a suitable constant $\gam$) and shows that
$\tpsiqqfw(\xxr)$ is \aex{} constant.
Hence, if $n_1,\dots,n_m$ are integers, not all 0, then thus the Fourier
coefficient
$\bigpar{\tpsiqqfw}\sphat(n_1,\dots,n_m,0,\dots,0)=0$.
However, it follows easily from \eqref{tpsiqq} and symmetry
that this Fourier coefficient is a
positive multiple of 
the Fourier coefficient 
$\bigpar{\tpsiqfw}\sphat(n_1,\dots,n_m)$.
Hence $\bigpar{\tpsiqfw}\sphat(n_1,\dots,n_m)=0$ when
$(n_1,\dots,n_m)\neq(0,\dots,0)$, and thus $\tpsiqfw$ is \aex{} constant;
it follows from \eqref{rpsi} that the constant must be $p\eef$.
By the proof of \refT{TA} (or by \refL{LW} and \refT{TA}), this implies
$W=p$ a.e. Consequently, \eqref{rpsi} for disjoint  $\aar$ with
$\gl(A_i)=\ga_i$ is a \pqrp, and thus so is $\ctP(F;\garr)$.
\end{proof}

\begin{example}[multicuts]
Consider the case $F=K_2$. Then the sum \eqref{nr} is
the number of edges with endpoints in two different sets $U_i$ and $U_j$;
we can call this a \emph{multicut}.
By \refT{Tr}, as proved already by \citet{ShapiraYuster:hyper} (see also
\citet{HuangLee}), 
the corresponding multicut property 
$\ctP(K_2;\garr)$ is a \pqr{} property
for any $(\garr)\neq(1/r,\dots,1/r)$. 
However,
$\ctP(K_2;1/r,\dots,1/r)$ is not \pqr{}, which is shown by the same 
counterexamples as for the case $r=2$ in \refE{EK2}.
\end{example}

If \refConj{Conj1} holds, then 
$\ctP(K_2;1/r,\dots,1/r)$ is essentially the only case when 
$\ctP(F;\garr)$ is not \pqr{}.


\section{Less parts than vertices}\label{Sless}

As said in \refR{R3}, 
it is interesting to study 
the subgraph counts
$N(F,G_n;\uum)$ and $\tN(F,G_n;\uum)$ 
also in situations
with other restrictions on 
the subsets $\uum$ than the ones considered above.
In particular, we may consider the case when the sets $U_i$ may be repeated, but
otherwise are disjoint. 
(We may also consider even more general situations when sets $U_i$ may
overlap partly in some prescribed ways, but that will not be treated here.)
This suggests the following general formulation:

Let $r\ge1$ and let $m_1,\dots,m_r$ be given non-negative integers
with 
$m_1+\dots+m_r=m=|F|$, and consider for a sequence of disjoint subsets
$U_1,\dots,U_r$ of $V(G)$, the following three subgraph counts:
\begin{romenumerate}[-10pt]
\item \label{lessa}
$N(F,G;U_1^{m_1},\dots,U_r^{m_r})$, defined as
$N(F,G;U_1,\dots,U_r)$ where the subset $U_i$ is repeated $m_i$ times.
(For a given labelling of $F$.)
\item 
$\tN(F,G;U_1^{m_1},\dots,U_r^{m_r})$, defined as
the average of $N(F,G;U_1^{m_1},\dots,U_r^{m_r})$ over all labellings
of $F$.
This equals, up to the constant symmetry factor $|\aut(F)|\prod_i m_i!/m!$, 
the number of copies of $F$ in $G$ that have exactly $m_i$ vertices in $U_i$.
(For each such copy of $F$, there are $\prod_i m_i!$ labellings of
$F$ for which it is counted, and the total number of labellings of $F$ is
$m!/|\aut(F)|$.) 
\item 
$\xN(F,G;U_1,\dots,U_r;m_1,\dots,m_r)$
defined as the further average of \\
$\tN(F,G;U_1^{m_1},\dots,U_r^{m_r})$
over all permutations of $m_1,\dots,m_r$.
\end{romenumerate}

We then define properties $\cP_{m_1,\dots,m_r}(F;\garr)$,
$\ctxP{m_1,\dots,m_r}(F;\garr)$ and
$\cxxP{m_1,\dots,m_r}(F;\garr)$ in analogy with \refD{DP}, 
considering all families of disjoint $\uur$ with $|U_i|=\floor{\ga_i|G_n|}$.
If $F=K_m$, then $\tN=N$ and thus
$\cP_{m_1,\dots,m_r}=\ctxP{m_1,\dots,m_r}$, but in general we do not know
any implication, \cf{} \refR{RX}.

\begin{example}
Note first that this formulation includes the problems studied earlier in
the paper:
\begin{alphenumerate}[-10pt]
\item 
For $r=m$ and $m_1=\dots=m_r=1$, we recover the main subject of the paper,
see \refS{Smain}.
(In this case, $\xN=\tN$.)
\item 
For $r>m$ and 
$m_i=1$ for $1\le i\le m$,
$m_i=0$ for $m+1\le i\le r$,
$\xN$ equals, up to
an unimportant constant factor, the sum \eqref{nr} studied in \refS{Smore}.
Thus
$\cxxP{m_1,\dots,m_r}(F;\garr)=\ctP(F;\garr)$.
\item 
For $r=1$ (and thus $m_1=m$), we consider $N(F,G;U,\dots,U)$ as in
\citet{SS:nni} (where $|U|$ is unspecified, see \refT{TSS}),
\citet{Shapira} and \citet{Yuster}.
\end{alphenumerate}  
\end{example}

The new case of main interest in the formulation above is 
thus $1<r<m$, with $2\le m_i <m$ for some $i$; thus some set $U_i$ is
repeated, but all are not equal.
In the remainder of this section, we consider a simple, but hopefully
typical, example of this, \viz{} $m=3$, $r=2$ and  $(m_1,m_2)=(2,1)$.

Thus,
assume that $m=|F|=3$.
For $\ga,\gb>0$ with $\ga+\gb\le1$,  the properties
$\cPz(F;\ga,\gb)$ and  $\ctPz(F;\ga,\gb)$ mean that
\eqref{l0} and \eqref{l0t}, respectively, hold for all $U_1,U_2,U_3$ with
$U_1=U_2$ but disjoint from $U_3$, and $|U_1|=\floor{\ga|G_n|}$,
$|U_3|=\floor{\gb|G_n|}$. In the case $\ga+\gb=1$, we can equivalently
assume that $U_1=U_2=U$ and $U_3=V(G_n)\setminus U$ with $|U|=\floor{\ga|G_n|}$.
(For $F=K_3$, this means that we count triangles crossing the the cut
$(U,\,V(G_n)\setminus U)$, with exactly two vertices in $U$.)
Are these properties \pqr?

The analogue of \refL{LA1c} holds, and thus we can as in \refL{Lequiv}
transfer the problem to graphons and the properties defined by \eqref{e33}
or \eqref{e34} for all $A_1,A_2,A_3$ with $A_1=A_2$ and disjoint from $A_3$,
and $\gl(A_1)=\ga$, $\gl(A_3)=\gb$.

Consider first $\ctPz(F;\ga,\gb)$.
In the case $\ga+\gb<1$, we have the following results, similar to the ones
above. 

\begin{lemma}\label{LQ}
  Let $\ga,\gb>0$ with $\ga+\gb<1$.
Suppose that $f:\oi^3\to\bbC$ is a symmetric integrable function
such that 
\begin{equation}\label{3sw}
\int_{A \times A\times B} f=0  
\end{equation}
for all disjoint subsets $A$ and $B$ of\/ $\oi$
such that
$\gl(A)=\ga$ and $\gl(B)=\gb$.
Then 
\begin{equation}\label{3fgae}
  f(\xxm)=0
\qquad a.e.
\end{equation}
\end{lemma}
\begin{proof}
 A minor variation of the proof of \refL{L3+}, using \citet[Lemma 7.6]{SJ234}.
We omit the details.
\end{proof}

\begin{theorem}\label{Tab}
Let $F$ be a graph with $|F|=3$ and $e(F)>0$,
 let $\ga,\gb>0$ with $\ga+\gb<1$
and let $0<p\le1$.
Then 
$\ctPz(F;\ga,\gb)$ is a \qr{} property.
\end{theorem}
\begin{proof}
 Using \refL{LQ}, we argue as in the proof of \refT{TA} in
 \refS{Salg}. 
\end{proof}

The case $\ga+\gb=1$, and thus $B=A\comp$ in \eqref{3sw}, is more
intricate,  and therefore more interesting.
We note first that the counterexample in \refL{Lcounter} shows that
\refL{LQ} does not hold for $\ga=1-\gb=2/3$.
In fact, if $f(x_1,x_2,x_3)=g(x_1)+g(x_2)+g(x_3)$ with $\intoi g=0$ and
$|A|=\ga$, 
then
\begin{equation}
  \begin{split}
\int_{A \times A\times A\comp} f(x_1,x_2,x_3)
&=2\ga(1-\ga)\int_A g + \ga^2\int_{A\comp} g
\\&
=2\ga(1-\ga)\int_A g - \ga^2\int_{A} g
=\ga(2-3\ga)\int_A g,
  \end{split}
\end{equation}
which vanishes for every such $A$ if $\ga=2/3$.

Moreover, there is another counterexample for $\ga=1-\gb=1/3$:
Now consider 
$f(x_1,x_2,x_3)=g(x_1,x_2)+g(x_1,x_3)+g(x_2,x_3)$ for a 
symmetric function $g$ on $\oi^2$ 
such that $\intoi g(x,y)\dd y=0$ for every $x$.
Then
\begin{equation}
  \begin{split}
\int_{A \times A\times A\comp} f(x_1,x_2,x_3)
&=(1-\ga)\int_{A^2} g + 2\ga\int_{A\times A\comp} g
\\&
=(1-\ga)\int_{A^2} g - 2\ga\int_{A^2} g
=(1-3\ga)\int_{A^2} g,
  \end{split}
\end{equation}
which vanishes if $\ga=1/3$.

We conjecture that these are the only counterexamples.
\begin{conjecture}\label{Conj2}
Let $\ga\in(0,1)$ and 
suppose that 
$f:\oi^3\to\bbC$ is a symmetric integrable function such that
$\int_{A\times A\times A\comp} f=0$ for every $A\subset\oi$
  with $\gl(A)=\ga$.
  \begin{romenumerate}[-15pt]
  \item 
If\/ $\ga\notin\set{\frac13,\frac23}$, then $f=0$ a.e.

  \item 
If\/ $\ga=\frac13$, then 
$f(x_1,x_2,x_3)=g(x_1,x_2)+g(x_1,x_3)+g(x_2,x_3)$ \aex{} for a 
symmetric function $g$ on $\oi^2$ 
such that $\intoi g(x,y)\dd y=0$ for every $x$.

  \item 
If\/ $\ga=\frac23$, then 
$f(x_1,x_2,x_3)=g(x_1)+g(x_2)+g(x_3)$ \aex{}
for a  function $g$ on $\oi$ 
such that $\intoi g(x)\dd x=0$.
  \end{romenumerate}
\end{conjecture}

We leave this (and extensions to $m>3$) as an open problem.
Note that if this conjecture holds, then 
\refT{Tab} holds also for $\ga+\gb=1$, provided $\ga\neq\frac13,\frac23$,
by the same proof as above.
For $\ga=1-\gb=\frac23$ we would have the same situation as in 
Lemmas \ref{Lmain} and \ref{Lalg}; from the discussion in \refS{Spf1}
follows that \refT{Tab} would hold if $e(F)\ge2$ ($P_2$ or $K_3$), but not
for $e(F)\le1$ ($K_2\cup K_1$ and the trivial empty graph $K_1\cup K_1\cup
K_1$). 
(Recall that we only consider $m=3$, as an example.)

For $\ga=1-\gb=\frac13$, even if the conjecture holds, it would lead to 
further open problems:
First, 
is there an analogue of 
Theorems \ref{TD0} and \ref{TD} for this case, showing that
if the property is not \qr, then there is a 2-type graphon counterexample?
(This seems likely if \refConj{Conj2} holds, using a suitable analogue of
\refL{LD} for this case.)
Secondly, analysis of a possible 2-type graphon counterexample would lead to a
different algebraic problem than the one in \refS{Salg}; we leave the
formulation and investigations of this as another open problem.

\begin{problem}
  Solve these problems for the case $\gb=1-\ga$, with particular attention
  to the cases $\ga=\frac13$ and $\frac23$, in particular for $F=K_3$ 
  (crossing triangles). Moreover, consider extensions for $m>3$.
\end{problem}

\begin{remark}\label{Rrepresentation}
Note that the set of functions satisfying the condition of Lemma \ref{L3},
\ref{Lmain} or \ref{LQ}, or \refConj{Conj2}, 
is invariant under all \mpb{s} of $\oi$.
This suggest the following approach, where we consider only square
integrable functions so that we can use Hilbert space theory.
Let, for $0\le k\le m$, $H^{m,k}$ be the subspace of 
$L^2(\oi^m)$ consisting of all functions $f$ such that the Fourier coefficient
$\hat f(n_1,\dots,n_m)$ vanishes unless exactly $k$ indices $n_1,\dots,n_m$
are non-zero. (In particular, $H^{m,0}$ is the space of constant functions.)
Let further 
$L^2\symm(\oi^m)$ be the subspace of symmetric functions in
$L^2(\oi^m)$, and let
$H^{m,k}\symm := H^{m,k}\cap L^2\symm(\oi^m)$.
Then 
\begin{equation}
 L^2\symm(\oi^m)=\bigoplus_{k=0}^m H\symm^{m,k}
\end{equation}
and each subspace $H\symm^{m,k}$ is invariant under \mpb{s} of $\oi$.
We conjecture that every closed 
subspace of $L^2\symm(\oi^m)$ invariant under all
\mpb{s} of $\oi$ is of the form 
$\bigoplus_{k\in A} H\symm^{m,k}$ for some set $A\subseteq\set{0,\dots,m}$.

If this holds, it is easy to verify \refConj{Conj2}. 

In support of this conjecture, note that a discrete analogue holds:
Let $N\ge m>0$ and consider the set $X_{N,m}$ of $m$-tuples of distinct elements
of $[N]$. If $N\ge 2m$, then 
the natural representation of the symmetric group $S_N$ in the 
$\binom Nm$-dimensional space of all symmetric functions on $X_{N,m}$
has $m+1$ irreducible components, which correspond to the sets
$H\symm^{m,k}$ above.
(This is easily verified by a calculation with the characters of these
representations. We omit the details.)
\end{remark}

Finally, for the property $\cPz(F;\ga,\gb)$, for a directed graph $F$ with
$|F|=3$, we have the same problems as before (unless $F=K_3$), see
\refR{RX}.
Consider for example $F=P_3$.
We may note that in \refL{LQ}, it suffices that $f$ is symmetric in the
first two variables; this implies by the argument above that if $F=P_3$ with
the central vertex labelled 3, then $\cPz(F;\ga,\gb)$ is \qr{}
(since then $\psifw$ is symmetric in the first two variables). However,
this argument fails for the other labellings of $P_3$.
The case $\ga+\gb=1$ seems even more complicated.

\begin{problem}
Is  $\cPz(P_3;\ga,\gb)$ a \qrp{} for any $\ga,\gb>0$ with $\ga+\gb<1$, for
any labelling of $P_3$?
Does this hold for $\ga+\gb=1$?
\end{problem}


\newcommand\AAP{\emph{Adv. Appl. Probab.} }
\newcommand\JAP{\emph{J. Appl. Probab.} }
\newcommand\JAMS{\emph{J. \AMS} }
\newcommand\MAMS{\emph{Memoirs \AMS} }
\newcommand\PAMS{\emph{Proc. \AMS} }
\newcommand\TAMS{\emph{Trans. \AMS} }
\newcommand\AnnMS{\emph{Ann. Math. Statist.} }
\newcommand\AnnPr{\emph{Ann. Probab.} }
\newcommand\CPC{\emph{Combin. Probab. Comput.} }
\newcommand\JMAA{\emph{J. Math. Anal. Appl.} }
\newcommand\RSA{\emph{Random Struct. Alg.} }
\newcommand\ZW{\emph{Z. Wahrsch. Verw. Gebiete} }
\newcommand\DMTCS{\jour{Discr. Math. Theor. Comput. Sci.} }

\newcommand\AMS{Amer. Math. Soc.}
\newcommand\Springer{Springer-Verlag}
\newcommand\Wiley{Wiley}

\newcommand\vol{\textbf}
\newcommand\jour{\emph}
\newcommand\book{\emph}
\newcommand\inbook{\emph}
\def\no#1#2,{\unskip#2, no. #1,} 
\newcommand\toappear{\unskip, to appear}

\newcommand\urlsvante{\url{http://www.math.uu.se/~svante/papers/}}
\newcommand\arxiv[1]{\url{arXiv:#1.}}
\newcommand\arXiv{\arxiv}

\def\nobibitem#1\par{}


\begin{thebibliography}{99}


\bibitem[Borgs, Chayes, \Lovasz, S\'os and Vesztergombi(2008)]{BCLSV1}
C. Borgs, J.~T. Chayes, L. \Lovasz, V.~T. S\'os \& K. Vesztergombi,
Convergent sequences of dense graphs I: Subgraph
frequencies, metric properties and testing,
\emph{Advances in Math.} {\bf 219} (2008), 1801--1851.


\bibitem[Borgs, Chayes, \Lovasz, S\'os and Vesztergombi(2012)]{BCLSV2}
C. Borgs, J.~T. Chayes, L. \Lovasz, V.~T. S\'os \& K. Vesztergombi,
Convergent sequences of dense graphs II. Multiway cuts and statistical physics. 
\emph{Ann. of Math. (2)} \textbf{176} (2012), no. 1, 151--219. 

\bibitem[Chung and Graham(1992)]{ChungG}
F. R. K. Chung \& R. L. Graham,
Maximum cuts and quasirandom graphs.  
\emph{Random graphs, Vol. 2 (Pozna\'n, 1989)},  23--33, 
Wiley, New York, 1992. 

\bibitem[Chung, Graham and Wilson(1989)]{ChungGW:quasi}
F. R. K. Chung, R. L. Graham \&  R. M. Wilson,
Quasi-random graphs.  
\jour{Combinatorica}  \vol9  (1989),  no. 4, 345--362. 



\bibitem[Huang and Lee(2012)]{HuangLee}
H. Huang \& C. Lee, 
Quasi-randomness of graph balanced cut properties.
\emph{Random Structures  Algorithms}
\vol{41} (2012), no. 1,  124--145.


\bibitem[Janson(2011)]{SJ234}
S. Janson,
Quasi-random graphs and graph limits.
\emph{European J. Combin.} \vol{32} (2011), 1054--1083.

\bibitem[\Lovasz{}(2012)]{Lovasz}
L. Lov{\'a}sz, 
\newblock {\em Large Networks and Graph Limits}.
\newblock American Mathematical Society, Providence, RI, 2012.

\nobibitem[\Lovasz{} and S\'os(2008)]{LSos}
L. \Lovasz{} \& V.~T. S\'os,
Generalized quasirandom graphs.
\emph{J. Combin. Theory B.}
\vol{98} (2008), no. 1, 146--163.

\bibitem[\Lovasz{} and Szegedy(2006)]{LSz}
L. \Lovasz{} \& B. Szegedy,
Limits of dense graph sequences.
\emph{J. Comb. Theory B} \vol{96} (2006), 933--957.

\bibitem[\Lovasz{} and Szegedy(2007)]{LSz:Sz}
L. \Lovasz{} \& B. Szegedy,
Szemer\'edi's lemma for the analyst.  
\jour{Geom. Funct. Anal.}  \vol{17}  (2007),  no. 1, 252--270.  

\bibitem[Petrov(2013)]{Petrov}
F. Petrov, 
General removal lemma.
\arxiv{1309.3795v1}

\bibitem[Shapira(2008)]{Shapira}
A. Shapira,
Quasi-randomness and the distribution of copies of a fixed graph. 
\emph{Combinatorica} \vol{28} (2008), 735--745. 


\bibitem[Shapira and Yuster(2010)]{ShapiraYuster}
A. Shapira \& R. Yuster,
The effect of induced subgraphs on quasi-randomness. 
\emph{Random Struct Algorithms} \vol{36} (2010), 90--109.


\bibitem[Shapira and Yuster(2012)]{ShapiraYuster:hyper}
A. Shapira \& R. Yuster, 
The quasi-randomness of hypergraph cut properties. 
\emph{Random Structures Algorithms} \vol{40} (2012), no. 1, 105--131.

\bibitem[Simonovits and S\'os(1991)]{SS:Sze}
M. Simonovits \& V. T. S\'os, 
Szemer\'edi's partition and quasirandomness.  
\emph{Random Structures Algorithms}  \vol2  (1991),  no. 1, 1--10. 

\bibitem[Simonovits and S\'os(1997)]{SS:nni}
M. Simonovits \& V. T. S\'os, 
Hereditarily extended properties, quasi-random graphs and not
necessarily induced subgraphs.  
\emph{Combinatorica}  \vol{17}  (1997),  no. 4,
577--596.  

\bibitem[Simonovits and S\'os(2003)]{SS:ind}
M. Simonovits \& V. T. S\'os, 
Hereditary extended properties, quasi-random graphs and induced
subgraphs. 
\emph{Combin. Probab. Comput.}  \vol{12}  (2003),  no. 3, 319--344.


\bibitem[Thomason(1987)]{Thomason87a}
A. Thomason, 
Pseudorandom graphs. 
\inbook{Random graphs '85 (Pozna\'n, 1985)}, 307--331, 
North-Holland, Amsterdam, 1987.

\bibitem[Thomason(1987)]{Thomason87b}
A. Thomason,  
Random graphs, strongly regular graphs and pseudorandom graphs. 
\inbook{Surveys in Combinatorics 1987 (New Cross, 1987)}, 173--195, 
London Math. Soc. Lecture Note Ser. \vol{123}, 
Cambridge Univ. Press, Cambridge, 1987. 

\bibitem[Yuster(2008)]{Yuster}
R. Yuster,
Quasi-randomness is determined by the distribution of copies of a
fixed graph in equicardinal large sets.
\inbook{Approximation, Randomization and Combinatorial Optimization},
596--601,
Lecture Notes in Comput. Sci. \vol{5171}, Springer, Berlin, 2008. 




\end{thebibliography}
\end{document}